\documentclass[leqno, 10pt, a4paper]{amsart}
\usepackage{amsmath}
\usepackage{amssymb, latexsym, mathrsfs, a4wide}

\newtheorem{theorem}{Theorem}[section]
\newtheorem{proposition}[theorem]{Proposition}

\newtheorem{step-main}{Step}

\theoremstyle{remark} 
\newtheorem{remark}[theorem]{Remark}
\newtheorem{example}[theorem]{Example}

\theoremstyle{definition}
\newtheorem{definition}[theorem]{Definition}

\newcommand{\C}{\ensuremath{\mathbb{C}}} 
\newcommand{\N}{\ensuremath{\mathbb{N}}} 
 
\newcommand{\R}{\ensuremath{\mathbb{R}}}

\newcommand{\cK}{\mathcal{K}}

\newcommand{\op}[1]{\operatorname{#1}}

\newcommand{\acou}[2]{\ensuremath{\langle #1 , #2 \rangle}} 

\newcommand{\Res}{\ensuremath{\op{Res}}}

\newcommand{\Tr}{\ensuremath{\op{Tr}}}

\newcommand{\Ric}{\op{Ric}}

\newcommand{\psido}{$\Psi$DO}
\newcommand{\psidos}{$\Psi$DOs}

\newcommand{\ord}{\op{ord}}

\newcommand{\pp}{\textup{pp}}

\begin{document}
\title{The Logarithmic Singularities of the Green Functions of the Conformal Powers of the Laplacian}
 \author{Rapha\"el Ponge}
 \address{Department of Mathematical Sciences, Seoul National University, Seoul, South Korea}
 \email{ponge.snu@gmail.com}

  \thanks{The research of this paper was partially supported by NSERC discovery grant 341328-07 (Canada), JSPS grant 
  in aid  30549291(Japan), research resettlement and foreign faculty research grants from Seoul National University, 
  and NRF grant for basic research 2013R1A1A2008802 (South Korea)}
 \begin{abstract}
Green functions play an important role in conformal geometry. In this paper, we explain how to compute explicitly the logarithmic singularities of the Green functions of the conformal powers of the Laplacian. These operators include the Yamabe and Paneitz operators, as well as the conformal fractional powers of the Laplacian arising from scattering theory 
for Poincar\'e-Einstein metrics. The results are formulated in terms of Weyl conformal invariants arising from the 
ambient metric of Fefferman-Graham. As applications we obtain  characterizations in terms of Green functions of locally 
conformally flat manifolds and a spectral theoretic characterization of the conformal class of the round sphere.
\end{abstract}

 \maketitle

\numberwithin{equation}{section}

\section*{Introduction}
Motivated by the analysis of the singularity of the Bergman kernel of a strictly pseudoconvex domain $\Omega\subset \C^{n}$, Fefferman~\cite{Fe:PITCA} 
launched the program of determining \emph{all} local invariants of a strictly pseudoconvex CR structure. This program was subsequently 
extended to deal with local invariants of other parabolic geometries, including conformal geometry~\cite{FG:CI, BEG:ITCCRG}. It has 
since found connections with various areas of mathematics and mathematical physics such as geometric PDEs, geometric 
scattering theory and conformal field theory. For instance, the Poincar\'e-Einstein metric of 
Fefferman-Graham~\cite{FG:CI, FG:AM} was a main impetus for the AdS/CFT correspondance. 

Green functions of conformally invariant operators plays a fundamental role in conformal geometry. 
Parker-Rosenberg~\cite{PR:ICL} computed the logarithmic singularity of Yamabe operator in low dimension. 
In~\cite{Po:LSSKLICCRS, Po:Clay} it was shown that the logarithmic singularities of Green 
functions of conformally invariant \psidos\ are linear 
combinations of Weyl conformal invariants. Those invariants are obtained from complete metric constructions of the covariant derivatives of the 
curvature tensor of ambient metric of Fefferman-Graham~\cite{FG:CI, FG:AM}. The approach of~\cite{PR:ICL} was based on results of 
Gilkey~\cite{Gi:ITHEASIT} on heat invariants for Laplace-type operators. It is not clear how to extend this approach to higher 
order GJMS operators, leave aside conformal fractional powers of the Laplacian. 

Exploiting the invariant theory for conformal structures, the main result of this paper is an explicit, and surprinsingly simple, formula for the logarithmic singularities of 
the Green functions of the conformal powers of the Laplacian in terms of Weyl conformal invariants obtained from the 
heat invariants of the Laplace operator (Theorem~\ref{thm:Main.main}). Here by conformal powers we mean the operators of 
Graham-Jenne-Mason-Sparling~\cite{GJMS:CIPLIE} and, more generally, the conformal fractional powers of the 
Laplacian~\cite{GZ:SMCG}. These operators include the Yamabe and Paneitz operators. 

Granted this result, it becomes straightforward to compute the logarithmic singularities of the Green functions of the 
conformal powers of the Laplacian from the sole knowledge of the heat invariants of the Laplace operators (see 
Theorem~\ref{thm:Main.weight4} and Theorem~\ref{thm:Main.weight6}). These results have several important consequences. 

In dimension $n\geq 5$ the logarithmic singularity of Green function of the $k$-th  conformal power of the Laplacian with $k=\frac{n}{2}-2$ 
is of special interest since this is a scalar multiple of the norm-square of the Weyl tensor. As a result 
we see that the vanishing of this logarithmic singularity is equivalent to local conformal-flatness (Theorem~\ref{thm:App.local-comformal-flat}). 
This can be seen as version in conformal geometry of the conjecture of Radamanov~\cite{Ra:CBCnBK} on the vanishing of the logarithmic singularity 
of the Bergman kernel. This also enables us to obtain a spectral theoretic characterization of the conformal class of the round sphere 
amount compact simply connected manifolds of dimension~$\geq 5$ (see Theorem~\ref{thm:App.spectral-caract-sphere}). 

The idea behind the proof of Theorem~\ref{thm:Main.main} is the following. In the Riemannian case, there is a  
simple relationship between the logarithmic singularities of the Green functions of the powers of the Laplace operator 
and its heat invariant (see Eq.~(\ref{eq:Green.log-sing-heat})). We may expect that by some analytic continuation of the signature this relationship 
still pertains in some way in non-Riemannian signature. The GJMS operators are obtained from the powers of the 
Laplacian associated to the ambient metric (which has Lorentzian signature), so there ought to be some relations between the logarithmic singularities of the Green functions of the conformal powers of the Laplace operator and its heat invariants. 

One way to test this conjecture is to look at the special case of Ricci-flat metrics. In this case the computation  
follows easily from the Riemannian case. The bulk of the proof then is to show that the Ricci-flat case implies the 
general case. Thus, the analytic extension of the signature is replaced by the new principle that in order to prove an 
equality between conformal invariants it is enough to prove it for Ricci-flat metrics. 

This ``Ricci-flat principle'' is actually fairly general, and so the results should also for other type of conformally 
invariant operators that are conformal analogues of elliptic covariant operators. In particular, it should hold for 
conformal powers of the Hodge Laplacian on forms~\cite{BG:CIODFCGQ, Va:ACHEERLFAHRS} and the  conformal powers of the square of the Dirac 
operator~\cite{GMP:EISZOTCCCHM,GMP:CBPSMB}. 

In addition, this approach can be extended to the setting of CR geometry. In particular, in~\cite{Po:LSSKLICCRS, 
Po:Clay} it was shown that the logarithmic singularities of the Green functions of CR invariant hypoelliptic 
$\Psi_{H}$DOs are linear combinations of Weyl CR invariants. We may also expect relate the logarithmic singularities of 
the Green functions of the CR invariant powers of the sub-Laplacian~\cite{GG:CRPS} to the heat invariants of the 
sub-Laplacian~\cite{BGS:HECRM}. However, the heat invariants for the sub-Laplacian are much less known than that for the 
Laplace operator. In particular, in order to recapture the Chern tensor requires computing the 3rd 
coefficient in the heat kernel asymptotics for the sub-Laplacian. In fact, for our purpose it would be enough to carry out the computation in the special case of circle bundles over Ricci-flat K\"ahler manifolds.

This paper is organized as follows. In Section~\ref{sec:Heat}, we recall the main facts on the heat kernel 
asymptotics for the Laplace operator. In Section~\ref{sec:Bergman}, we recall the geometric description of the singularity of the Bergman kernel of 
a strictly pseudoconvex complex domain and the construction of local CR invariants by means of Fefferman's ambient 
K\"ahler-Lorentz metric. In Section~\ref{sec:Ambient}, we recall the construction of the Fefferman-Graham's ambient metric and GJMS 
operators. In Section~\ref{sec:Conformal-Invariants}, we recall the construction of local conformal invariants by means of the ambient metric. 
In Section~\ref{sec:scattering}, we explain the construction of conformal fractional powers of the Laplacian and its connection with 
scattering theory and the Poincar\'e-Einstein metric. In Section~\ref{sec:Green}, we gather basic facts about Green functions 
and their relationship with heat kernels. In Section~\ref{sec:Main}, we state the main result and derive various consequences. In Section~\ref{sec:Outline}, we give an outline of the 
proof of the main result. 

\subsection*{Acknowledgements}
{\small It is a pleasure to thank Pierre Albin, Charles Fefferman, Rod Gover, Robin Graham, Colin Guillarmou, Kengo Hirachi, Dmitry Jakobson, 
Andreas Juhl, Andras Vasy, and Maciej Zworski for various discussions related to the subject matter of this paper. 
Part of the research of this paper was carried out during visits to Kyoto University, McGill University, MSRI, UC Berkeley, and the University of Tokyo. 
The author wish to thank these institutions for their hospitality.} 

 \section{The Heat Kernel of the Laplace Operator}\label{sec:Heat}
The most important differential operator attached to a closed Riemannian manifold $(M^{n},g)$ is the Laplace operator,
\begin{equation*}
    \Delta_{g}u= \frac{-1}{\sqrt{\det g(x)}}\sum_{i,j}\partial_{i}\left( g^{ij}(x)\sqrt{\det g(x)} \, \partial_{j}u\right).
\end{equation*}
This operator lies at the interplay between Riemannian geometry and elliptic theory. On the one hand, a huge amount of 
geometric information can be extracted from the analysis of the Laplace operator. For instance, if 
$0=\lambda_{0}(\Delta_{g})<\lambda_{1}(\Delta_{g})\leq \lambda_{2}(\Delta_{g})\leq \cdots $ are the eigenvalues of 
$\Delta_{g}$ counted with multiplicity, then Weyl's Law asserts that, as $k\rightarrow \infty$,  
\begin{equation}
    \lambda_{k}\left(\Delta_{g}\right)\sim (ck)^{\frac{2}{n}}, \qquad  c:=(4\pi)^{-\frac{n}{2}}\Gamma\left( 
    \frac{n}{2}+1\right)^{-1} \op{Vol}_{g}(M), 
    \label{eq:Heat.Weyl-law}
\end{equation}where $\op{Vol}_{g}(M)$ is the Riemannian volume of $(M, g)$. On the other hand, Riemannian geometry is used to describe singularities or asymptotic 
behavior of solutions of 
PDEs associated to the Laplace operator. This aspect is well illustrated by the asymptotics of the heat kernel of $\Delta_{g}$. 

Denote by $e^{-t\Delta_{g}}$, $t>0$, the semigroup generated by $\Delta_{g}$. That is, 
$u(x,t)=\left(e^{-t}\Delta_{g}u_{0}\right)(x)$ is solution of the heat equation,
\begin{equation*}
            ( \partial_{t}+\Delta_{g})u(x,t)=0, \qquad u(x,t)=u_{0}(x) \quad \text{at $t=0$}.
            \label{eq:Heat.heat-equation}
\end{equation*}The heat kernel $K_{\Delta_{g}}(x,y;t)$ is the kernel function of the heat semigroup, 
\begin{equation*}
    e^{-t\Delta_{g}}u(x)=\int_{M} K_{\Delta_{g}}(x,y;t)u(y)v_{g}(y),
\end{equation*}where $v_{g}(y)=\sqrt{g(x)}|dx|$ is the Riemannian volume density. Equivalently, $K_{\Delta_{g}}(x,y;t)$ provides 
us with a fundamental solution for the heat equation~(\ref{eq:Heat.heat-equation}). Furthermore, 
\begin{equation}
 \sum e^{-t\lambda_{j}(\Delta_{g})}=   \Tr e^{-t\Delta_{g}}= \int_{M}K_{\Delta_{g}}(x,x;t)u(x)v_{g}(x). 
 \label{eq:Heat.heat-trace}
\end{equation}
 
 \begin{theorem}[\cite{ABP:OHEIT, Gi:CELEC, Mi:ERM}]\label{thm:Heat.heat-kernel-asymptotics}  
Let $(M^{n},g)$ be a closed manifold. Then
\begin{enumerate}
    \item As $t \rightarrow 0^{+}$, 
   \begin{equation}
        K_{\Delta_{g}}(x,x;t) \sim (4\pi t)^{-\frac{n}{2}}\sum_{j\geq 0} t^{j}a_{2j}(\Delta_{g};x),
        \label{eq:Heat.heat-kernel-asymptotics}
    \end{equation}
   where the asymptotics holds with respect to the Fr\'echet-space topology of $C^{\infty}(M)$. 

    \item  Each coefficient $a_{2j}(\Delta_{g};x)$ is a local Riemannian invariant of weight~$2j$.
\end{enumerate}
\end{theorem}
\begin{remark}
   There is an asymptotics similar to~(\ref{eq:Heat.heat-kernel-asymptotics}) for the heat kernel of any selfadjoint differential operators with 
    positive-definite principal symbol (see~\cite{Gi:ITHEASIT, Gr:AEHE}). 
\end{remark}

\begin{remark}
    A \emph{local Riemannian invariant}
is a (smooth) function $I_{g}(x)$ of $x\in M$ and the metric tensor $g$ which, in any local 
coordinates, has an universal expression of the form,
\begin{equation}
    I_{g}(x)= \sum a_{\alpha\beta}\left(g(x)\right) \left(\partial^{\beta}g(x)\right)^{\alpha},
    \label{eq:Heat.Riemann-invariant}
\end{equation}where the sum is finite and the $a_{\alpha\beta}$ are smooth functions on $\op{Gl}_{n}(\R)$ that are 
\emph{independent} of the choice of the local coordinates. We further say that $I_{g}(x)$ has weight $w$ when
\begin{equation*}
        I_{\lambda^{2} g}(x)=\lambda^{-w}I_{g}(x) \quad \forall \lambda>0.
\end{equation*}These definitions continue to make sense for pseudo-Riemannian structures of nonpositive 
signature. Note also that with our convention for the weight, the weight is always an even nonnegative integer. 
\end{remark}

Examples of Riemannian invariants are provided by complete metric contractions of the tensor products 
of covariant derivatives of the curvature tensors 
$R_{ijkl}=\acou{R(\partial_{i},\partial_{j})\partial_{k}}{\partial_{l}}$. Namely, 
\begin{equation*}
         \op{Contr}_{g}(\nabla^{k_{1}}R\otimes \cdots \otimes \nabla^{k_{l}}R). 
\end{equation*}Such an invariant has weight $w=(k_{1}+\cdots + k_{l})+2l$ and is called a \emph{Weyl 
Riemannian invariant}. Up to a sign factor, the only 
Weyl Riemannian invariants of weight $w=0$ and $w=2$ are the 
constant function $1$ and the scalar curvature $\kappa:=R^{ij}_{\mbox{~~~}ji}$ respectively. In weight $w=4$ we obtain the following four invariants,
\begin{equation*}
    \kappa^{2},\qquad |\Ric |^{2}:=\Ric^{ij}\Ric_{ij}, \qquad |R|^{2}:=R^{ijkl}R_{ijkl}, \qquad \Delta_{g}\kappa,
\end{equation*}where $\Ric_{ij}:=R^{k}_{\mbox{~~}ijk}$ is the Ricci tensor. In weight $w=6$ they are 17 such invariants; 
the only invariants that do not involve the Ricci tensor are
\begin{equation*}
   |\nabla R|^{2}:=\nabla^{m}R^{ijkl}\nabla_{m}R_{ijkl}, \qquad  
   R_{ij}^{\mbox{~~}\mbox{~~}kl}R_{\mbox{~~}\mbox{~~}pq}^{ij}R_{\mbox{~~}\mbox{~~}kl}^{pq}, 
   \qquad  R_{ijkl} R^{i\mbox{~~}k}_{\;p\mbox{~}q} R^{pjql}.
\end{equation*}

\begin{theorem}[Atiyah-Bott-Patodi~\cite{ABP:OHEIT}]\label{thm:Heat.invariant-theory}  
Any local Riemannian invariant of weight $w$, $w\in 2\N_{0}$, is a universal linear combination of Weyl Riemannian invariants of same weight.
\end{theorem}
\begin{remark}
By using normal coordinates the proof is reduced to determining all invariant polynomial of the orthogonal group 
$\op{O}(n)$ (or $\op{O}(p,q)$ in the pseudo-Riemannian case). They are determined thanks to Weyl's invariant theory for 
semisimple groups. 
\end{remark}

Combining Theorem~\ref{thm:Heat.heat-kernel-asymptotics}  and Theorem~\ref{thm:Heat.invariant-theory}  we obtain the 
following structure theorem for the heat invariants.  

\begin{theorem}[Atiyah-Bott-Patodi~\cite{ABP:OHEIT}]\label{eq:Heat.invariant-heat-kernel-asymptotics}
    Each heat invariant $a_{2j}(\Delta_{g};x)$ in~(\ref{eq:Heat.heat-kernel-asymptotics}) is a universal linear combination of 
    Weyl Riemannian invariants of weight~$2j$. 
\end{theorem}
\begin{remark}
    There is a similar structure result for the heat invariants of any selfadjoint elliptic covariant differential operator with positive principal symbol. 
\end{remark}

In addition, the first few of the heat invariants are computed. 

\begin{theorem}[\cite{BGM:SVR, MS:CEL, Gi:SGRM}]\label{thm:Heat.coefficients}
The first four heat invariants $a_{2j}(\Delta_{g};x)$, $j=0,\ldots,3$ are given by the following formulas,   
\begin{equation*}
          a_{0}(\Delta_{g};x)=1, \quad  a_{2}(\Delta_{g};x)= -\frac{1}{6}\kappa,  \quad 
          a_{4}(\Delta_{g};x)=\frac{1}{180}|R|^{2}-\frac{1}{180}|\op{Ric}|^{2}+\frac{1}{72}\kappa^{2}-\frac{1}{30}\Delta_{g}\kappa,
 \end{equation*}
        \begin{multline}
            a_{6}(\Delta_{g};x)= \frac{1}{9\cdot7!} \left( 81|\nabla 
            R|^{2}+64R_{ij}^{\mbox{~~}\mbox{~~}kl}R_{\mbox{~~}\mbox{~~}pq}^{ij}R_{\mbox{~~}\mbox{~~}kl}^{pq}+352
             R_{ijkl} R^{i\mbox{~~}k}_{\;p\mbox{~}q} R^{pjql}
            \right)\\  + \textup{Weyl Riemannian invariants involving the Ricci tensor}.
            \label{eq:Heat.formula-I3}
        \end{multline}
\end{theorem}

\begin{remark}
 We refer to Gilkey's monograph~\cite{Gi:ITHEASIT} for the full formula for $a_{6}(\Delta_{g};x)$ and formulas for the 
 heat invariants of various Laplace type operators, including the Hodge Laplacian on forms. We mention that Gilkey's 
 formulas involve a constant multiple of $-\acou{R}{\Delta_{g}R}=g^{pq}R_{ijkl}R_{ijkl;pq}$, but this Weyl Riemannian 
 invariant is a linear 
 combination of other Weyl Riemannian invariants. In particular, using the Bianchi identities we find that, modulo Weyl Riemannian involving the Ricci tensor, 
 \begin{equation*}
   - \acou{R}{\Delta_{g}R }= R_{ij}^{\mbox{~~}\mbox{~~}kl}R_{\mbox{~~}\mbox{~~}pq}^{ij}R_{\mbox{~~}\mbox{~~}kl}^{pq}+4
             R_{ijkl} R^{i\mbox{~~}k}_{\;p\mbox{~}q} R^{pjql}. 
 \end{equation*}This relation is incorporated into~(\ref{eq:Heat.formula-I3}). 
\end{remark}

\begin{remark}
 Polterovich~\cite{Po:HIRM} established formulas for \emph{all} the heat invariants $a_{j}(\Delta_{g};x)$ in terms of the Riemannian distance function.
\end{remark}

\begin{remark}
 Combining the equality $ a_{0}(\Delta_{g};x)=1$ with~(\ref{eq:Heat.heat-trace}) and~(\ref{eq:Heat.heat-kernel-asymptotics}) shows that, as $t\rightarrow 0^{+}$, 
\begin{equation*}
    \sum e^{-t\lambda_{j}(\Delta_{g})} \sim (4\pi t)^{-\frac{n}{2}} \int_{M}v_{g}(x)=(4\pi t)^{-\frac{n}{2}}\op{Vol}_{g}(M).
\end{equation*}We then can apply Karamata's Tauberian theorem to recover Weyl's Law~(\ref{eq:Heat.Weyl-law}). 
    \end{remark}

\begin{remark}
 A fundamental application of the Riemannian invariant theory of the heat kernel asymptotics is the proof of  the local index 
 theorem by Atiyah-Bott-Patodi~\cite{ABP:OHEIT} (see also~\cite{Gi:ITHEASIT}). 
\end{remark}

\section{The Bergman Kernel of a Strictly Pseudoconvex Domain}\label{sec:Bergman}
A fundamental problem in several complex variables is to find local computable biholomorphic invariants of a strictly pseudoconvex domain $\Omega\subset \C^n$. One approach to this issue is to look at the boundary singularity of the Bergman metric of $\Omega$ or equivalently its Bergman kernel. 
Recall that  the Bergman kernel is the kernel function of the orthogonal projection,
   \begin{gather*}
      B_{\Omega}:L^{2}(\Omega)\longrightarrow L^{2}(\Omega)\cap \op{Hol}(\Omega),\\
      B_{\Omega}u(z)=\int K_{\Omega}(z,w)u(w)dw. 
  \end{gather*}
For instance, in the case of the unit ball $\mathbb{B}^{2n}=\{|z|<1\}\subset \C^{n}$ we have
 \begin{equation*}
     B_{\mathbb{B}^{2n}}(z,w)=\frac{n!}{\pi^{n}}(1-z\overline{w})^{-(n+1)}.
 \end{equation*}

The Bergman kernel lies at the interplay of complex analysis and differential geometry. On the 
one hand, it provides us with the reproducing kernel of the domain $\Omega$ and it plays a fundamental role in the 
analysis of the $\overline{\partial}$-Neuman problem on $\Omega$. On the other hand, it provides us with a  
biholomorphic invariant K\"ahler metric, namely, the Bergman metric,
\begin{equation*}
      ds^{2}= \sum \frac{\partial^{2}}{\partial z^{j} \partial z^{\overline{k}}} \log B(z,z) dz^{j} dz^{\overline{k}}.
\end{equation*}

In what follows we let $\rho$ be a defining function of $\Omega$ so that $\Omega=\{\rho<0\}$ and $i\partial\overline{\partial} \rho>0$. 
  
\begin{theorem}[Fefferman, Boutet de Monvel-Sj\"ostrand]
    Near the boundary $\partial\Omega=\{\rho=0\}$, 
    \begin{equation*}
        K_{\Omega}(z,z)=\varphi(z)\rho(z)^{-(n+1)}-\psi(z) \log \rho(z),
    \end{equation*}where $\varphi(z)$ and $\psi(z)$ are smooth up to $\partial\Omega$.
\end{theorem} 
  
Motivated by the analogy with the heat asympototics~(\ref{eq:Heat.heat-kernel-asymptotics}), where the role of the time 
variable $t$ is played by 
the defining function $\rho(z)$, Fefferman~\cite{Fe:PITCA} launched the program of giving a geometric description 
of the singularity of the Bergman kernel similar to the description provided by Theorem~\ref{thm:Heat.heat-kernel-asymptotics}  
and Theorem~\ref{eq:Heat.invariant-heat-kernel-asymptotics} for the heat kernel asymptotics. 

A first issue at stake concerns the choice of the defining function. We would like to make a biholomorphically invariant choice 
of defining function. This issue is intimately related to the complex Monge-Amp\`ere equation on $\Omega$: 
       \begin{equation}
           J(u):=(-1)^{n}\det 
           \begin{pmatrix}
               u  & \partial_{z^{\overline{k}}}u \\
               \partial_{z^{j}}u & \partial_{z^{j}}\partial_{z^{\overline{k}}}u
           \end{pmatrix},
            \qquad 
           u_{\left| \partial \Omega\right.}=0.
           \label{eq:Bergman.Monge-Ampere}
    \end{equation}
A solution of the Monge-Amp\`ere equation is unique and biholomorphically invariant in the sense that, given 
any bilohomorphism $\Phi:\Omega\rightarrow \Omega$, we have
\begin{equation*}
    u\left( \Phi(z)\right) = \left| \det \Phi'(z)\right|^{\frac{2}{n+1}}u(z).
\end{equation*}

We mention the following important results concerning the Monge-Amp\`ere equation. 

\begin{theorem}[Cheng-Yau~\cite{CY:RMAE}] Let $\Omega \subset \C^{n}$ be a strictly pseudoconvex domain. 
\begin{enumerate}
 \item  There is a unique exact solution $u_{0}(z)$ of the complex Monge-Amp\`ere equation~(\ref{eq:Bergman.Monge-Ampere}). 
 
 \item The solution $u_{0}(z)$ is 
 $C^{\infty}$ on $\Omega$ and belongs to $C^{n+\frac{3}{2}-\epsilon}\left(\overline{\Omega}\right)$ for all 
 $\epsilon>0$.

 \item  The metric ${\displaystyle ds^{2}=\sum \frac{\partial^{2}}{\partial z^{j} \partial z^{\overline{k}}}
 \log \left( u_{0}(z)^{-(n+1)} \right) dz^{j} dz^{\overline{k}}}$ is a K\"ahler-Einstein metric on $\Omega$.
\end{enumerate}   
\end{theorem}

\begin{theorem}[Lee-Melrose~\cite{LM:BBCMAE}]
Let $\Omega \subset \C^{n}$ be a strictly pseudoconvex domain with defining function $\rho(z)$. Then, near the boundary 
$\partial\Omega=\{\rho=0\}$, the Cheng-Yau solution to the Monge-Amp\`ere equation has a behavior of the form, 
\begin{equation}
u_{0}(z) \sim \rho(z)\sum_{k\geq 0} \eta_k(z) \left( \rho(z)^{n+1} \log \rho(z) \right)^k,    
    \label{eq:Bergman.asymptotic-CY}
\end{equation}
where the functions $\eta_{k}(z)$ are smooth up to the boundary. 	
\end{theorem}

The Cheng-Yau solution is not smooth up to the boundary, but if we only seek for asymptotic solutions then we do get 
smooth solutions. 

\begin{theorem}[Fefferman~\cite{Fe:MAEBKGPCD}]
Let $\Omega \subset \C^{n}$ be a strictly pseudoconvex domain with defining function $\rho(z)$. Then there are functions in 
$C^{\infty}\left( \overline{\Omega}\right)$ that are solutions to the asymptotic Monge-Amp\`ere equation, 
       \begin{equation}
           J(u)=1+\op{O}(\rho^{n+1}) \ \text{near $\partial \Omega$}, \qquad u_{\left| \partial \Omega\right.}=0.
           \label{eq:Bergman.Monge-Ampere-approximate}
       \end{equation}    
\end{theorem}
\begin{remark}
As it can be seen from~(\ref{eq:Bergman.asymptotic-CY}), the error term $\op{O}(\rho^{n+1})$ cannot be improved in general if we seek for a 
smooth approximate solution. 
\end{remark}

Any smooth solution $u(z)$ of~(\ref{eq:Bergman.Monge-Ampere-approximate}) is a defining function for $\Omega$ and it is asymptotically biholomorphically 
invariant in the sense that, for any biholomorphism $\Phi:\Omega \rightarrow \Omega$, 
\begin{equation*}
    u\left( \Phi(z)\right) = \left| \det \Phi'(z)\right|^{\frac{2}{n+1}}u(z) +\op{O}(\rho(z)^{n+1})
\end{equation*}

The complex structure of $\Omega$ induces on the boundary $\partial\Omega$ a strictly pseudoconvex CR structure. As a consequence of a well-known result of Fefferman~\cite{Fe:BKBMPCD} there is 
one-to-one correspondance between biholomorphisms of $\Omega$ and CR diffeomorphisms of $\partial\Omega$. Therefore, boundary values of 
biholomorphic invariants of $\Omega$ gives rise to CR invariants of $\partial\Omega$. We refer to~\cite{Fe:PITCA} for the precise 
definition of a local CR invariant. Let us just mention that a local CR invariant $I(z)$ has weight $\omega$ when, for any biholomorphism 
$\Phi:\Omega\rightarrow \Omega$, we have
\begin{equation*}
    I\left( \Phi(z)\right) = \left| \det \Phi'(z)\right|^{-\frac{w}{n+1}}I(z) \qquad \text{on $\partial\Omega$}. 
\end{equation*}


\begin{proposition}[Fefferman~\cite{Fe:PITCA}]
Let $\Omega \subset \C^{n}$ be a strictly pseudoconvex domain. Then
\begin{enumerate}
    \item  Near the boundary $\partial\Omega$, 
\begin{equation*}
 K_{\Omega}(z,z)= u(z)^{-(n+1)}\sum_{j=0}^{n+1}I_{u}^{(j)}(z)u(z)^{j}+\op{O}\left(\log u(z)\right),  
\end{equation*}where $u(z)$ is a smooth solution of~(\ref{eq:Bergman.Monge-Ampere-approximate}). 

   \item  For each $j$, the boundary value of $I_{u}^{(j)}(z)$ is a local CR invariant of weight $2j$.
\end{enumerate}
\end{proposition}


This leads us to the geometric part of program of Fefferman: determining all CR invariants of strictly pseudoconvex 
domain so as to have an analogue of Theorem~\ref{thm:Heat.invariant-theory} for CR invariants. However, the CR invariant theory 
is much more involved than the Riemannian invariant theory. The geometric problem reduces to the invariant theory for a parabolic 
subgroup $P\subset \op{SU}(n+1,1)$. However, as $P$ is not semisimple, Weyl's invariant theory does not apply anymore. 
The corresponding invariant theory was developed by Fefferman~\cite{Fe:PITCA} and Bailey-Eastwood-Graham~\cite{BEG:ITCCRG}. 

The Weyl CR invariants are constructed as follows. Let $u(z)$ be a smooth approximate 
solution to the Monge-Amp\`ere equation in the sense of~(\ref{eq:Bergman.Monge-Ampere-approximate}). On the ambient space $\C^{*}\times \Omega$ consider the potential, 
\begin{equation*}
    U(z_{0},z)=|z_{0}|^{2}u(z), \qquad (z_{0},z)\in \C^{*}\times \Omega.
\end{equation*}Using this potential we define the K\"ahler-Lorentz metric, 
\begin{equation*}
    \tilde{g}=\sum_{0\leq j,k\leq n}  \frac{\partial^{2}U}{\partial z^{j} \partial z^{\overline{k}} }dz^{j} 
      dz^{ \overline{k}}.
\end{equation*}
A biholomorphism $\Phi:\Omega\rightarrow \Omega$ acts on the ambient space by
\begin{equation*}
    \tilde{\Phi}(z_{0},z)=\left(\left| \det \Phi'(z)\right|^{-\frac{1}{n+1}},\Phi(z)\right), \qquad (z_{0},z)\in \C^{*}\times \Omega.
\end{equation*}This is the transformation law under biholomorphisms for $\cK_{\Omega}^{\frac{1}{n+1}}$, where $\cK_{\Omega}$ is the canonical line bundle 
of $\Omega$. 

\begin{theorem}[Fefferman~\cite{Fe:MAEBKGPCD}] \mbox{~}\label{thm:Bergman.Kalher-Lorentz-metric}
\begin{enumerate}
    \item    The K\"ahler-Lorentz metric $\tilde{g}$ is asymptotically Ricci flat and invariant under biholomorphisms 
    in the sense that
    \begin{equation*}
       \tilde{ \Phi}^{*}\tilde{g}=\tilde{g}+\op{O}\left( \rho(z)^{n+1}\right) \quad \text{and} \quad 
       \Ric(\tilde{g})=\op{O}\left( \rho(z)^{n}\right) \qquad \text{near $\C^{*}\times \partial \Omega$},
    \end{equation*}where $\Phi$ is any biholomorphism of $\Omega$.

    \item  The restriction of $\tilde{g}$ to $S^{1}\times \partial\Omega$ is a Lorentz metric $g$ whose conformal class is 
    invariant under biholomorphisms.
\end{enumerate}
  \end{theorem}

\begin{remark}
    The above construction of the ambient metric is a special case of the ambient metric construction associated to a 
    conformal structure due to Fefferman-Graham~\cite{FG:CI, FG:AM} (see also Section~\ref{sec:Ambient}).  
\end{remark}

\begin{remark}
    We refer to~\cite{Le:FMPHI} for an intrinsic construction of the Fefferman circle bundle $S^{1}\times \partial\Omega$ and its
    Lorentz metric $g$ for an arbitrary nondegenerate CR manifold. 
\end{remark}

It follows from Theorem~\ref{thm:Bergman.Kalher-Lorentz-metric} that, given a Weyl Riemannian invariant $I_{\tilde{g}}(z_{0},z)$ constructed out of the covariant 
derivatives of the curvature tensor of $\tilde{g}$, the boundary value  of $I_{\tilde{g}}(z_{0},z)$ is a local CR invariant provided that 
it does not involved covariant derivatives of too high order. This condition is fullfilled if 
$I_{\tilde{g}}(z_{0},z)$ has weight $w\leq 2n$, in which case its boundary value is a CR invariant of same weight. Such an invariant is called a \emph{Weyl CR invariant}. 
We observe that the 
Ricci-flatness of the ambient metric kills all Weyl invariants involving the Ricci tensor of $\tilde{g}$. Therefore, 
there are much fewer Weyl CR invariants than Riemannian Weyl invariants.

The following results are the analogues of Theorem~\ref{thm:Heat.invariant-theory}  and Theorem~\ref{eq:Heat.invariant-heat-kernel-asymptotics} for the Bergman kernel. 
    
    \begin{theorem}[Fefferman~\cite{Fe:PITCA}, Bayley-Eastwood-Graham~\cite{BEG:ITCCRG}] 
    Any local CR invariant of weight $w\leq 2n$ is a 
    linear combination of Weyl CR invariants of same weight. 
\end{theorem}

\begin{theorem}[Fefferman~\cite{Fe:PITCA}, Bayley-Eastwood-Graham~\cite{BEG:ITCCRG}] 
Let $\Omega \subset \C^{n}$ be a strictly pseudoconvex domain. Then, near the boundary $\partial\Omega$, 
\begin{equation*}
 K_{\Omega}(z,z)= u(z)^{-(n+1)}\sum_{j=0}^{n+1}I_{u}^{(j)}(z)u(z)^{j}+\op{O}\left(\log u(z)\right),  
\end{equation*}where $u(z)$ is a smooth solution of~(\ref{eq:Bergman.Monge-Ampere-approximate}) and, for $j=0,\ldots,n+1$, 
\begin{equation*}
    I_{u}^{(j)}(z)=J^{(j)}_{\tilde{g}}(z_{0},z)_{|z_{0}=1},
\end{equation*}where $J^{(j)}_{\tilde{g}}(z_{0},z)$ is a linear combination of complete metric contractions of weight $w$ 
of the covariant derivatives of the K\"ahler-Lorentz metric $\tilde{g}$. 
\end{theorem}

\begin{remark}
We refer to~\cite{Hi:CBILDBK} for an invariant description of the logarithmic singularity $\psi(z)$ of the Bergman kernel. 
\end{remark}


\section{Ambient Metric and Conformal Powers of the Laplacian}\label{sec:Ambient}
Let $(M^{n},g)$ be a Riemannian manifold. In dimension $n=2$ the Laplace operator is conformally 
invariant in the sense that, if we make a conformal change of metric $\hat{g}=e^{2\Upsilon}g$, $\Upsilon\in 
     C^{\infty}(M,\R)$, then 
\begin{equation*}
    \Delta_{\hat{g}}=e^{-2\Upsilon} \Delta_{g}.
\end{equation*}

This conformal invariance breaks down in dimension $n\geq 3$. A conformal version of the Laplacian is provided by the Yamabe 
operator, 
\begin{gather*}
    P_{1,g}:=\Delta_{g}+\frac{n-2}{4(n-1)}\kappa, \\
    P_{1,e^{2\Upsilon}g}=e^{-\left(\frac{n}{2}+1\right)\Upsilon} P_{1,g}e^{\left(\frac{n}{2}-1\right)\Upsilon} \qquad \forall \Upsilon\in 
     C^{\infty}(M,\R). 
\end{gather*}
The operator of Paneitz~\cite{Pa:QCCDOAPRM} is a conformally invariant operator with principal part $\Delta_{g}^{2}$. Namely, 
  \begin{gather*}
    P_{2,g}:=\Delta_{g}^{2}+\delta  V d+\frac{n-4}{2}\left\{
    \frac{1}{2(n-1)}\Delta_{g}R_g+\frac{n}{8(n-1)^{2}}R_g^{2}-2|P|^{2}\right\},\\ 
P_{2,e^{2\Upsilon}g}=e^{-\left(\frac{n}{2}+2\right)\Upsilon} P_{2,g}e^{\left(\frac{n}{2}-2\right)\Upsilon} \qquad 
\forall \Upsilon\in 
     C^{\infty}(M,\R).
\end{gather*}Here $P_{ij}=\frac{1}{n-2}(\op{Ric_g}_{ij}-\frac{R_g}{2(n-1)}g_{ij})$ is the Schouten-Weyl tensor and $V$ 
is the tensor $V_{ij}=\frac{n-2}{2(n-1)} R_g g_{ij}-4P_{ij}$ acting on 1-forms (i.e., $V(\omega_{i}dx^{i})=(V_{i}^{~j}\omega_{j})dx^{i}$). 

More generally, Graham-Jenne-Mason-Sparling~\cite{GJMS:CIPLIE} constucted higher order conformal powers of the 
Laplacian by using the ambient metric of Fefferman-Graham~\cite{FG:CI, FG:AM}. This metric extends Fefferman's 
K\"ahler-Lorentz metric to any ambient space associated to a conformal structure. It is constructed as 
follows. 

Consider the ray-bundle, 
\begin{equation*}
    G:=\bigsqcup_{x\in M} \left\{t^{2}g(x); \ t>0\right\}\subset S^{2}T^{*}M,
\end{equation*}where $S^{2}T^{*}M$ is the bundle of symmetric $(0,2)$-tensors. We note that $G$ depends only the 
conformal class $[g]$. Moreover, on $G$ there is a natural family of dilations $\delta_{s}$, $s>0$, given by 
\begin{equation}
    \delta_{s}(\hat{g})=s\hat{g} \qquad \forall \hat{g}\in G.
    \label{eq:Confinv.dilationsds}
\end{equation}

Let $\pi:G\rightarrow M$ be the canonical submersion of $G$ onto $M$ (i.e., the restriction to $G$ of the canonical 
submersion of the bundle $S^{2}T^{*}M$). Let $d\pi^{t}:S^{2}T^{*}M\rightarrow S^{2}T^{*}G$ be the differential of $\pi$ on 
symmetric $(0,2)$-tensors. Then $G$ carries a canonical symmetric $(0,2)$-tensor $g_{0}$ defined by
\begin{equation}
    g_{0}(x,\hat{g})=d\pi(x)^{t}\hat{g} \qquad \forall (x,\hat{g})\in G.
    \label{eq:Ambient.canonical-tensor}
\end{equation}
The datum of the representative metric $g$ defines a fiber coordinate $t$ on $G$ such that $\hat{g}=t^{2}g(x)$ for all 
$(x,\hat{g})\in G$. If $\{x^{j}\}$ are local coordinates on $M$, then $\{x^{j},t\}$ are local coordinates on $G$ and in 
these coordinates the tensor $g_{0}$ is given by
\begin{equation*}
    g_{0}=t^{2}g_{ij}(x)dx^{i}dx^{j},
\end{equation*}where $g_{ij}$ are the coefficients of the metric tensor $g$ in the local coordinates $\{x^{j}\}$. 

The \emph{ambient space} is  the $(n+2)$-dimensional manifold
\begin{equation*}
   \tilde{G}:=G\times (-1,1)_{\rho}, 
\end{equation*}where we denote by $\rho$ the variable in $(-1,1)$. We identify $G$ with the 
hypersurface $G_{0}:=\{\rho=0\}$. We also note that the dilations $\delta_{s}$ in~(\ref{eq:Confinv.dilationsds}) lifts to 
dilations on $\tilde{G}$. 

\begin{theorem}[Fefferman-Graham~\cite{FG:CI, FG:AM}] 
    Near $\rho=0$ there is a Lorentzian metric $\tilde{g}$, called the ambient metric, which is defined up to infinite order in 
    $\rho$ when $n$ is odd,  and up to order $ \frac{n}{2}$ when $n$ is even, such that
  \begin{enumerate}
      \item[(i)]  In local coordinates $\{t,x^{j},\rho\}$,
 \begin{equation*}
    \tilde{g}=2\rho (dt)^{2}+t^{2}g_{ij}(x,\rho)+2tdtd\rho,
\end{equation*}where $g_{ij}(x,\rho)$ is a family of symmetric $(0,2)$-tensors such that 
$\left.g(x,\rho)\right|_{\rho=0}=g_{ij}(x)$.
  
      \item[(ii)] The ambient metric is asymptotically Ricci-flat, in the sense that 
      \begin{equation}
          \Ric(\tilde{g})=\left\{ 
          \begin{array}{cc}
              \op{O}(\rho^{\infty}) & \text{if $n$ is odd},  \\
              \op{O}(\rho^{\frac{n}{2}}) & \text{if $n$ is even}.
          \end{array}\right. \label{eq:Confinv.Ricci-flatness-ambient}
      \end{equation}
  \end{enumerate}
\end{theorem}
\begin{remark}
    We note that~(i) implies that
    \begin{equation}
       \tilde{g}_{\left| TG_{0}\right.}=g_{0} \qquad \text{and} \qquad \delta_{s}^{*}\tilde{g}=s^{2}\tilde{g} \quad 
       \forall s>0,
       \label{eq:Ambient.homogeneity-ambient-metric}
    \end{equation}where $g_{0}$ is the symmetric tensor~(\ref{eq:Ambient.canonical-tensor}).
 \end{remark}

\begin{remark}
The ambient metric is unique up to its order of definition near $\rho=0$.  Solving the equation~(\ref{eq:Confinv.Ricci-flatness-ambient}) leads us to a system of nonlinear PDEs for 
$g_{ij}(x,\rho)$. When $n$ is even, there is an obstruction to solve this system at infinite order which given by the 
a conformally invariant tensor. This tensor is the Bach tensor in dimension 4. Moreover, it vanishes on conformally Einstein metrics.  
\end{remark}

\begin{example}\label{ex:Ambient.Einstein}
    Suppose that $(M,g)$ is Einstein with $\Ric(g)=2\lambda(n-1)g$. Then 
\begin{equation*}
    \tilde{g}=2\rho (dt)^{2}+t^{2}(1+\lambda \rho)^{2}g_{ij}(x)+2tdtd\rho. 
\end{equation*}
\end{example}

Conformal powers of the Laplacian are obtained from the powers of the ambient metric Laplacian 
$\tilde{\Delta}_{\tilde{g}}$ on $\tilde{G}$ as follows. 

\begin{proposition}[Graham-Jenne-Mason-Sparling~\cite{GJMS:CIPLIE}]
Let $k\in \N_{0}$ and further assume $k\leq \frac{n}{2}$ when $n$ is even. Then
\begin{enumerate}
    \item  We define a operator $P_{k,g}:C^{\infty}(M)\rightarrow C^{\infty}(M)$ by
    \begin{equation}
        P_{k,g}u(x):= \left. t^{-\left(\frac{n}{2}+k\right)}\tilde{\Delta}^{k}_{\tilde{g}}\left(t^{\frac{n}{2}-k}\tilde{u}(x,\rho)\right)\right|_{\rho=0}, \qquad u \in 
        C^{\infty}(M),
        \label{eq:GJMS.definition-GJMS}
    \end{equation}where $\tilde{u}$ is any smooth extension of $u$ to $M\times \R$ (i.e., the r.h.s.\ above 
    is independent of the choice of  $\tilde{u}$). 

    \item  $P_{k,g}$ is a differential operator of order $2k$ with same leading part as $\Delta_{g}^{k}$.

    \item  $P_{k,g}$ is a conformally covariant differential operator such that
    \begin{equation}
        P_{k,e^{2\Upsilon}g}= e^{-\left(\frac{1}{2}n+k\right)\Upsilon}(P_{k, g})e^{\left(\frac{1}{2}n-k\right)\Upsilon}\qquad \forall \Upsilon \in 
        C^{\infty}(M,\R). 
        \label{eq:GJMS.conformal-invariance-Pkg}
    \end{equation}
\end{enumerate}
\end{proposition}

\begin{remark}
    The homogeneity of the ambient metric $\tilde{g}$ implies that the r.h.s.\ is idenpendent of $t$ and 
    obeys~(\ref{eq:GJMS.conformal-invariance-Pkg}).
\end{remark}

\begin{remark}
    The operator $P_{k,g}$ is called the $k$-th GJMS operator. For $k=1$ and $k=2$ we recover the Yamabe operator and Paneitz 
    operator respectively. 
 \end{remark}

\begin{remark}
     The GJMS operators are selfadjoint~\cite{FG:QCPM, GZ:SMCG}. 
\end{remark}

\begin{remark}
  The GJMS operator  $P_{k,g}$ can also be obtained as the obstruction to extending a function $u\in C^{\infty}(M)$ into a 
  homogeneous harmonic function on the ambient space.   In addition, we refer to~\cite{FG, GP:CIPLQCTC, Ju:FCCDOQCH, Ju} for 
  various features and properties of the GJMS operarors. 
\end{remark}

\begin{remark}
 When the dimension $n$ is even, the ambient metric construction is obstructed at finite order, and so the GJMS construction breaks 
 down for $k>\frac{n}{2}$. As proved by Graham~\cite{Gr:CIPLN} in dimension 4 for
$k=3$ and by Gover-Hirachi~\cite{GH:CIPLCNET} in general, there do not exist
conformally invariant operators with same leading part as
$\Delta^{k}_{g}$ for $k>\frac{n}{2}$ when $n$ is even. 
\end{remark}

\begin{remark}
The operator $P_{\frac{n}{2},g}$ is sometimes called the
\emph{critical GJMS operator}. Note that for $P_{\frac{n}{2}}$ the
transformation law~(\ref{eq:GJMS.conformal-invariance-Pkg}) becomes
\begin{equation*}
    P_{\frac{n}{2},e^{2\Upsilon}g}=e^{-n\Upsilon}P_{\frac{n}{2},g} \qquad 
\forall \Upsilon\in C^{\infty}(M,\R).
\end{equation*}
It is tempting to use the critical GJMS operator as a candidate for a subsitute to the Laplacian to prove a version of the 
Knizhnik-Polyakov-Zamolodchikov (KPZ) formula in dimension~$\geq 3$. We refer to~\cite{CJ:GFFKPZRR4} for results in this direction.      
\end{remark} 

\begin{example}[Graham \cite{Gr:AIM03, FG:AM}, Gover~\cite{Go:LOQCCEM}] \label{lem:GJMS.GJMS-Einstein}
 Suppose that $(M^{n},g)$ is a Einstein manifold with $\op{Ric}_{g}=\lambda(n-1)g$, $\lambda\in \R$. Then, for all 
 $k\in \N_{0}$, we have
\begin{equation}
    P_{k,g}=\prod_{1\leq j \leq k}\left(\Delta_{g}-\frac{1}{4}\lambda(n+2j-2)(n-2j)\right).
    \label{eq:GJMS.GJMS-Einstein}
\end{equation}
\end{example}

\section{Local Conformal Invariants} \label{sec:Conformal-Invariants}
In this section, we describe local scalar invariants of a conformal structure and explain how to construct them by 
means of the ambient metric of Fefferman-Graham~\cite{FG:CI, FG:AM}. 



\begin{definition}
  A local conformal invariant of weight $w$, $w\in 2\N_{0}$, is a local 
  Riemannian invariant $I_{g}(x)$ such that
  \begin{equation*}
      I_{e^{2\Upsilon} g}(x)=e^{-w\Upsilon(x)} I_{g}(x) \qquad \forall \Upsilon\in C^{\infty}(M,\R).
  \end{equation*} 
\end{definition}

The most important conformally invariant tensor is the \emph{Weyl tensor},
\begin{equation}
    W_{ijkl}=R_{ijkl}-(P_{jk}g_{il}+P_{il}g_{jk}-P_{jl}g_{ik}-P_{ik}g_{jl}), 
\end{equation}
where $P_{jk}=\frac{1}{n-2}(\Ric_{jk}-\frac{\kappa_{g}}{2(n-1)}g_{jk})$ is the Schouten-Weyl tensor. In particular, in 
dimension~$n\geq 4$ the Weyl tensor vanishes if and only if $(M,g)$ is locally conformally flat. Moreover, as the Weyl tensor is conformally invariant of weight 2, 
we get scalar conformal invariants by taking complete metric contractions of tensor products of the Weyl tensor. 

Let $I_{g}(x)$ be a local Riemannian invariant of weight $w$. By using the ambient metric, Fefferman-Graham~\cite{FG:CI, 
FG:AM} produced a recipe for constructing local conformal invariant from $I_{g}(x)$ as follows.\medskip 

\noindent \underline{Step 1:} Thanks to Theorem~\ref{thm:Heat.invariant-theory}  we know that $I_{g}(x)$ is a linear combination of complete metric 
contractions of covariant derivatives of the curvature tensor. Substituting into this complete metric contractions the ambient metric $\tilde{g}$ and the 
covariant derivatives of its curvature tensor we obtain a local Riemannian invariant $I_{\tilde{g}}(t,x,\rho)$ on the 
ambient metric space $(\tilde{G},\tilde{g})$. For instance, 
\begin{equation*}
    \op{Contr}_{g}(\nabla^{k_{1}}R\otimes \cdots \otimes \nabla^{k_{l}}R) \longrightarrow 
    \op{Contr}_{\tilde{g}}(\tilde{\nabla}^{k_{1}}\tilde{R}\otimes \cdots \otimes \tilde{\nabla}^{k_{l}}\tilde{R})
\end{equation*}where $\tilde{\nabla}$ is the ambient Levi-Civita connection and $\tilde{R}$ is the ambient curvature 
tensor.\medskip  

\noindent \underline{Step 2:} We define a function on $M$ by
\begin{equation}
    \tilde{I}_{g}(x):=\left. t^{-w}I_{\tilde{g}}(t,x,\rho)\right|_{\rho=0} \qquad \forall x \in M.
    \label{eq:Ambient.FG-rule}
\end{equation}If $n$ is odd this is always well defined. If $n$ is even this is well defined provided only derivatives of $\tilde{g}$ of not 
too high order are involved; this is the case if $w\leq n$.  Note also that thanks to the homogeneity of the ambient 
metric the r.h.s.\ of~(\ref{eq:Ambient.FG-rule}) is always independent of the variable $t$. 

\begin{proposition}[\cite{FG:CI, FG:AM}] 
The function  $\tilde{I}_{g}(x)$ is a local conformal invariant of weight $w$.
 \end{proposition}
\begin{remark}
    The fact that $\tilde{I}_{g}(x)$ transforms conformally under conformal change of metrics is a consequence of 
    the homogeneity of the ambient metric in~(\ref{eq:Ambient.homogeneity-ambient-metric}). 
\end{remark}

 We shall refer to the rule $I_{g}(x) \rightarrow \tilde{I}_{g}(x)$ as \emph{Fefferman-Graham's rule}. When $I_{g}(x)$ is a Weyl Riemannian invariant we shall call 
 $\tilde{I}_{g}(x)$ a \emph{Weyl conformal invariant}. 
 
 \begin{example}\label{ex:Ambient.FG-rule}
     Let us give a few examples of using Fefferman-Graham's rule (when $n$ is even we further assume that the 
     invariants $I_{g}(x)$ below have weight~$\leq n$).
     \begin{enumerate}
         \item[(a)] If $I_{g}(x)=  \op{Contr}_{g}(\nabla^{k_{1}}\Ric \otimes \nabla^{k_{2}}R\otimes \cdots \otimes \nabla^{k_{l}}R)$, then the 
         Ricci flatness of $\hat{g}$ implies that
         \begin{equation*}
             \tilde{I}_{g}(x)=0.
         \end{equation*}
     
         \item[(b)] (See~\cite{FG:CI, FG:AM}.) If $I_g(x)=\op{Contr}_{g}(R \otimes \cdots \otimes R)$ (i.e., no covariant derivatives are involved), then 
         $\tilde{I}_{g}(x)$ is obtained by substituting the Weyl tensor for the curvature tensor, that is, 
         \begin{equation*}
             \tilde{I}_{g}(x)= \op{Contr}_{g}(W \otimes \cdots \otimes W). 
         \end{equation*}
     
         \item[(c)] (See~\cite{FG:CI, FG:AM}.) For $I_{g}(x)=|\nabla R|^{2}$, we have
         \begin{equation*}
             \tilde{I}_{g}(x)=\Phi_{g}(x):=|V|^{2}+16\acou{W}{U}+16|C|^{2},
    \label{eq:Ambient.Phi-invariant}
\end{equation*}
where $C_{jkl}=\nabla_{l}P_{jk}-\nabla_{k}P_{jl}$ is the Cotton tensor and the tensors $U$ and $V$ are defined by
\begin{gather*}
    V_{mijkl}:=\nabla_{s}W_{ijkl}-g_{im}C_{jkl}+g_{jm}C_{ikl}-g_{km}C_{lij}+g_{lm}C_{kij}, \\
  U_{mjkl}:=\nabla_{m}C_{jkl}+g^{rs}P_{mr}W_{sjkl}.
\end{gather*}
     \end{enumerate}
 \end{example}

\begin{theorem}[Bailey-Eastwood-Graham~\cite{BEG:ITCCRG}, Fefferman-Graham~\cite{FG:AM}]\label{prop:Ambient.invariant-theory} 
    Let $w \in 2\N_{0}$, and further assume $w\leq n$ when 
    $n$ is even. Then every local conformal invariant  of weight $w$ is a linear combination of Weyl conformal 
   invariants of same weight.
\end{theorem}
 
\begin{remark}
  As in the Riemannian case, the strategy for the proof of Theorem~\ref{prop:Ambient.invariant-theory}  has two main steps. The first step is a 
  reduction of the geometric problem to the invariant theory for a parabolic subgroup $P\subset \op{O}(n+1,1)$. This reduction is 
 a consequence of the jet isomorphism theorem in conformal geometry, which is proved in~\cite{FG:AM}. The relevant invariant 
 theory for the parabolic subgroup $P$ is developped in~\cite{BEG:ITCCRG}.  
\end{remark}
 
\begin{remark}
    We refer to~\cite{GH:IAM} for a description of the local conformal invariants in even dimension beyond the critical 
    weight $w=n$.  
\end{remark}

\section{Scattering Theory and Conformal Fractional Powers of the Laplacian}\label{sec:scattering}
The ambient metric construction realizes a conformal structure as a hypersuface in the ambient space. An equivalent 
approach is to realize a conformal structure as the conformal infinity of an asymptotically hyperbolic Einstein (AHE) manifold. This 
enables us to use scattering theory to construct conformal fractional powers of the Laplacian. 

Let $(M^{n},g)$ be a compact Riemannian manifold. We assume that $M$ is the boundary of some manifold $\overline{X}$ with 
interior $X$. Note we always can realize the double $M\sqcup M$ as the boundary of $[0,1]\times M$. We assume that $X$ 
carries a AHE metric $g^{+}$. This means that
\begin{enumerate}
    \item  There is a definining function $\rho$ for $M=\partial X$ such 
that, near $M$,  the metric $g^{+}$ is of the form,
    \begin{equation*}
        g^{+}=\rho^{-2}g(\rho,x)+\rho^{-2}(d\rho)^{2},
    \end{equation*}where $g(\rho,x)$ is smooth up $\partial X$ and agrees with $g(x)$ at $\rho=0$.

    \item  The metric $g^{+}$ is Einstein, i.e., 
    \begin{equation*}
        \Ric(g^{+})=-ng^{+}.
    \end{equation*}
\end{enumerate}
In particular $(X,g^{+})$ has conformal boundary $(M,[g])$, where $[g]$ is the conformal 
class of $g$. 

The scattering matrix for an AHE
metric $g^{+}$ is constructed as follows. We denote by 
$\Delta_{g^{+}}$ the Laplacian on $X$ for the metric $g^{+}$. The continuous spectrum of $\Delta_{g^{+}}$ agrees with 
$[\frac{1}{4}n^{2},\infty)$ and its pure point spectrum $\sigma_{\pp}(\Delta_{g^{+}})$ is a finite subset of 
$\left(0,\frac{1}{4}n^{2}\right)$ (see~\cite{Ma:HCCCM, Ma:UCIEAHM}). Define
\begin{equation*}
    \Sigma= \left\{ s\in \C; \ \Re s\geq \frac{n}{2}, \ s\not\in \frac{n}{2}+\N_{0}, \ s(n-s)\not \in 
    \sigma_{\pp}(\Delta_{g^{+}})\right\}.
\end{equation*}
For $s \in \Sigma$, consider the eigenvalue problem,
\begin{equation*}
    \Delta_{g^{+}}-s(n-s)=0.
\end{equation*}For any $f\in C^{\infty}(M)$, the above eigenvalue problem as a unique solution of the form, 
\begin{equation*}
    u=F\rho^{n-s}+G\rho^{s}, \quad F,G\in C^{\infty}(\overline{X}), \quad F_{\left|\partial X\right.}=f.
\end{equation*}The scattering matrix is the operator $S_{g}(s):C^{\infty}(M)\rightarrow 
C^{\infty}(M)$ given by 
\begin{equation*}
    S_{g}(s)f:=G_{\left|M \right.} \qquad \forall f \in C^{\infty}(M). 
\end{equation*}
Thus $S_{g}(s)$ can be seen as a generalized Dirichlet-to-Neuman 
operator, where $F_{\left| \partial X\right.}$ represents the ``Dirichlet data'' and $G_{\left|\partial X\right.}$ represents 
the ``Neuman data''. 

\begin{theorem}[Graham-Zworski~\cite{GZ:SMCG}] For $z \in -\frac{n}{2}+\Sigma$ define
    \begin{equation*}
        P_{z,g}:=2^{2z}\frac{\Gamma(z)}{\Gamma(-z)}S_{g}\left( \frac{1}{2}n+z\right).
    \end{equation*}
   The family $(P_{z,g})$ uniquely extends to a holomorphic family of pseudodifferential operators $(P_{z,g})_{\Re z 
   >0}$ in such way that
   \begin{enumerate}
       \item[(i)] The operator $P_{z,g}$ is a \psido\ of order $2z$ with same principal symbol as $\Delta_{g}^{z}$.  
   
       \item[(ii)] The operator $P_{z,g}$ is conformally invariant in the sense that
       \begin{equation}
           P_{z,e^{2\Upsilon}g}= e^{-\left(\frac{n}{2}+z\right)\Upsilon}(P_{z, 
           g})e^{\left(\frac{n}{2}-z\right)\Upsilon} \qquad \forall \Upsilon\in C^{\infty}(M,\R).
           \label{eq:Scattering.conformal-invariancePzg}
       \end{equation}

       \item[(iii)] The family $(P_{z,g})_{z \in \Sigma}$ is a holomorphic family of \psidos\ and has a meromorphic 
       extension to the half-space $\Re z>0$ with at worst finite rank simple pole singularities.  
          
       \item[(iv)] Let $k\in \N$ be such that $-k^{2}+\frac{1}{4}n^{2}\not \in \sigma_{\pp}(\Delta_{g^{+}})$. Then
       \begin{equation*}
           \lim_{z\rightarrow k}P_{z,g}=P_{k,g},
       \end{equation*}where $P_{k,g}$ is the $k$-th GJMS operator defined by~(\ref{eq:GJMS.definition-GJMS}). 
   \end{enumerate}
\end{theorem}
\begin{remark}
    Results of Joshi-S\'a Barreto~\cite{JS:ISAHM} show that each operator $P_{z,g}$ is a Riemannian invariant \psido\ in the 
    sense that all the homogeneous components are given by universal expressions in terms of the partial derivatives of 
    the components of the metric $g$ (see~\cite{Po:Clay} for the precise definition). 
\end{remark}

\begin{remark}
If $k\in \N$ is such that $\lambda_{k}:=-k^{2}+\frac{1}{4}n^{2} \in \sigma_{\pp}(\Delta_{g^{+}})$, then~(\ref{eq:Scattering.conformal-invariancePzg}) holds 
modulo a finite rank smoothing operator obtained as the restriction to $M$ of the orthogonal projection onto the 
eigenspace $\ker\left (\Delta_{g^{+}}-\lambda_{k}\right)$.
\end{remark}

\begin{remark}
    The analysis of the scattering matrix $S_{g}(s)$ by Graham-Zworski~\cite{GZ:SMCG} relies on the analysis of the 
    resolvent $\Delta_{g^{+}}-s(n-s)$ by Mazzeo-Melrose~\cite{MM:MERCSACNC}. Guillarmou~\cite{Gu:MPRAHM} established the meromorphic 
    continuation of the resolvent to the whole complex plane $\C$. There are only finite rank poles when the AHE 
    metric is even.  This gives the meromorphic continuation of the scattering matrix $S_{g}(z)$ and the operators 
    $P_{z,g}$ to the whole complex plane. We refer to~\cite{Va:MAAHSHERE, Va:ACHEERLFAHRS} for alternative approaches to these questions.  
\end{remark}

\begin{remark}
    We refer to~\cite{CD:FLCG} for an interpretation of the operators $P_{z,g}$ as fractional Laplacians in the sense of 
    Caffarelli-Sylvestre~\cite{CS:EPRFL}. 
\end{remark}

\begin{example}[Branson~\cite{Br:SIFDCS}]
    Consider the round sphere $\mathbb{S}^{n}$ seen as the boundary of the unit ball $\mathbb{B}^{n+1}\subset \R^{n+1}$ equipped with 
    its standard hyperbolic metric. Then
    \begin{equation*}
        P_{z,g}= 
        \frac{\Gamma\left(\sqrt{\Delta_{g}+\frac{1}{4}(n-1)^{2}}+1+z\right)}{\Gamma\left(\sqrt{\Delta_{g}+\frac{1}{4}(n-1)^{2}}+1-z\right)}, \qquad z\in\C. 
    \end{equation*}
    We refer to~\cite{GN:W0TLSCCHM} for an extension of this formula to manifolds with constant sectional curvature. 
\end{example}

A general conformal structure cannot always be realized as the conformal boundary of an AHE manifold. However, as 
showed by Fefferman-Graham~\cite{FG:CI, FG:AM} it always can be realized as a conformal boundary formally.  

Let $(M^{n},g)$ be a Riemannian manifold. Define $X=M\times (0,\infty)$ and let $\overline{X}=M\times [0,\infty)$ be the closure of $X$. 
We shall denote by $r$ the variable in $[0,\infty)$ and we identify $M$ with the 
boundary $r=0$. 

\begin{theorem}[Fefferman-Graham~\cite{FG:CI, FG:AM}] Assume $M$ has odd dimension. Then, near the boundary 
    $r=0$, there is a Riemannian metric $g^{+}$, called Poincar\'e-Einstein metric, which is defined up to infinite order in $r$ and satisfies the following properties:
\begin{enumerate} 
          \item[(i)]  In local coordinates $\{r,x^{j}\}$,
 \begin{equation*}
   g^{+}= r^{-2}(dr)^{2}+r^{-2}g_{ij}(x,\rho),
\end{equation*}where $g_{ij}(x,r)$ is a family of symmetric $(0,2)$-tensors such that 
$\left.g(x,r)\right|_{r=0}=g_{ij}(x)$.
  
      \item[(ii)] It is asymptotically Einstein in the sense that, near the boundary $r=0$, 
      \begin{equation}
          \Ric(g^{+})=-ng^{+}+  \op{O}(r^{\infty}).
          \label{eq:Scattering.AEinstein}
      \end{equation}
\end{enumerate}
\end{theorem}

\begin{remark}
    When the dimension $n$ is even there is also a Poincar\'e-Einstein metric which is defined up to order $n-2$ in $r$
     and satisfies the Einstein equation~(\ref{eq:Scattering.AEinstein}) up to an $\op{O}(r^{n-2})$ error. 
\end{remark}

\begin{remark}
    The constructions of the Poincar\'e-Einstein metric and ambient metric are equivalent. We can obtain either 
    metric from the other. For instance, the map $\phi: (r,x)\rightarrow \left(r^{-1},x,-\frac{1}{2}r^{2}\right)$ is a 
    smooth embedding of $X$ into the hypersurface,
\begin{equation*}
        \mathcal{H}^{+}= \left\{\tilde{g}(T,T)=-1\right\}= \left\{(t,x\rho);\ 2\rho t^{2}=-1\right\}\subset \tilde{G},
\end{equation*}
where $T=t\partial_{t}$ is the infinitesimal generator for the dilations~(\ref{eq:Confinv.dilationsds}). Then $g^{+}$ is obtained 
as the pullback to $X$ of $\tilde{g}_{\left|T\mathcal{H}\right.}$. In particular, the metric $g^{+}$ is even in the sense that 
the tensor $g(x,r)$ has a Taylor expansion near $r=0$ involving  \emph{even} powers of $r$ only. 
\end{remark}

\begin{example}
Assume that $(M,g)$ is Einstein, with $\Ric(g)=2\lambda(n-1)g$. Then 
\begin{equation*}
    g^{+}=r^{-2}(dr)^{2}+r^{-2}\left(1-\frac{1}{2}\lambda r^{2}\right)g_{ij}(x)dx^{i}dx^{j}. 
\end{equation*}
\end{example}

\begin{proposition}
    Assume $n$ is odd. Then there is an entire family $(P_{g,z})_{z \in\C}$ of \psidos\ on $M$ uniquely defined up to 
    smoothing operators such that
    \begin{enumerate}
        \item[(i)]  Each operator $P_{g,z}$, $z\in \C$, is a \psido\ of order $2z$ and has same principal symbol as 
        $\Delta_{g}^{z}$. 
    
        \item[(ii)]  Each operator $P_{g,z}$ is a conformally invariant \psido\ in the sense of~\cite{Po:Clay} in such 
        a way that, for all $\Upsilon\in C^{\infty}(M,\R)$, 
               \begin{equation}
           P_{z,e^{2\Upsilon}g}= e^{-\left(\frac{n}{2}+z\right)\Upsilon}(P_{z,g})e^{\left(\frac{n}{2}-z\right)\Upsilon} \qquad \bmod \Psi^{-\infty}(M),
           \label{eq:Scaterring.almost-conformal-invariancePzg}
       \end{equation}where $\Psi^{-\infty}(M)$ is the space of smoothing operators on $M$.
    
        \item[(iii)] For all $k\in \N$, the operator $P_{z,g}$ agrees at $z=k$ with the $k$-th GJMS operator $P_{k,g}$.
        
       \item[(iv)] For all $z \in \C$, 
       \begin{equation}
           P_{z,g}P_{-z,g}=1 \qquad \bmod \Psi^{-\infty}(M).
           \label{eq:Sattering.functional-equation-Pzg}
       \end{equation}
    \end{enumerate}
\end{proposition}

\begin{remark}
The equation~(\ref{eq:Sattering.functional-equation-Pzg}) is an immediate consequence of the functional equation for the scattering matrix, 
\begin{equation*}
    S_{g}(s)S_{g}(n-s)=1.
\end{equation*}
This enables us to extend the family $(P_{z,g})$ beyond the halfspace $\Re z>0$. 
\end{remark}

\begin{remark}
    When $n$ is even, we also can construct operators $P_{z,g}$ as above except that each operator $P_{z,g}$ is unique 
and satisfies~(\ref{eq:Scaterring.almost-conformal-invariancePzg}) modulo \psidos\ of order $z-2n$. 
\end{remark}

\begin{remark}
By using a polynomial continuation of the homogeneous components of the symbols of the GJMS operators, 
Petterson~\cite{Pe:CCPDO} also constructed a holomorphic family of \psidos\ satisfying the properties (i)--(iv) above. 
\end{remark}

\begin{example}[Joshi-S\'a Barreto~\cite{JS:ISAHM}]\label{ex:scattering.ricci-flat-Pzg}
    Suppose that $(M,g)$ is Ricci flat and compact. Then, for all $z \in \C$,
    \begin{equation*}
       P_{g,z}= \Delta_{g}^{z}\qquad \bmod \Psi^{-\infty}(M).
    \end{equation*}
\end{example}

As mentioned above, the constructions of the ambient metric and Poincar\'e-Einstein metric are equivalent. Thus, it 
should be possible to define the GJMS operators directly in terms of the Poincar\'e-Einstein metric, so as to interpret the GJMS 
operators as boundary operators. This interpretation is actually implicit in~\cite{GZ:SMCG}. Indeed, if we combine  the definition~(\ref{eq:GJMS.definition-GJMS}) 
of the GJMS operators with the formula on the 
2nd line from the top on page~113 of~\cite{GZ:SMCG}, then we arrive at the following statement.

\begin{proposition}[\cite{GJMS:CIPLIE}, \cite{GZ:SMCG}]
 Let $k\in \N$ and further assume $k\leq \frac{n}{2}$ when $n$ is even. Then, for all $u\in C^{\infty}(M)$, 
   \begin{equation*}
       P_{k,g}u= \biggl. r^{-\frac{1}{2}n-k} \prod_{j=0}^{k-1}\left( 
       \Delta_{g^{+}}+\left(-\frac{1}{2}n+k-2j\right)\left(\frac{1}{2}n+k-2j\right)\right)(r^{\frac{1}{2}n-k}u^{+})\biggr|_{r=0}, 
   \end{equation*}where $u^{+}$ is any function in $C^{\infty}(\overline{X})$ 
   which agrees with $u$ on the boundary $r=0$. 
\end{proposition}

\begin{remark}
 Guillarmou~\cite{Gu:PC} obtained this formula via scattering theory. We also refer to~\cite{GW:BCCCM} for a similar formula in terms of the tractor bundle.  In addition, using a similar approach as above, Hirachi~\cite{Hi:PC} derived an analogous formula for the CR GJMS operators of Gover-Graham~\cite{GG:CRPS}.  
\end{remark}

\section{Green Functions of Elliptic Operators}\label{sec:Green}
In this section,  we describe the singularities of the Green function of an elliptic operator and its closed 
relationship with the heat kernel asymptotics.

Let $(M^{n},g)$ be a Riemannian manifold. Let $P:C^{\infty}_{c}(M)\rightarrow 
C^{\infty}(M)$ be an elliptic pseudodifferential operator (\psido) of (integer) order $m>0$. Its \emph{Green function} 
$G_{P}(x,y)$ (when it exists) is the fundamental solution of $P$, that is,
\begin{equation}
    P_{x}G_{P}(x,y)= \delta(x-y).
    \label{eq:Green.Green-equation}
\end{equation}Equivalently, $G_{P}(x,y)$ is the inverse kernel of $P$, i.e., 
\begin{equation*}
    P_{x}\left(\int_{M}G_{P}(x,y)u(y)v_{g}(y)\right)=u(x). 
\end{equation*}

In general, $P$ may have a nontrivial kernel (e.g., the kernel of the Laplace operator $\Delta_{g}$ consists of 
constant functions). Therefore, by Green function we shall actually mean a solution of~(\ref{eq:Green.Green-equation}) modulo a 
$C^{\infty}$-error, i.e., 
\begin{equation*}
    P_{x}G_{P}(x,y)= \delta(x-y) \qquad \bmod C^{\infty}(M\times M).
\end{equation*}
This means that using $G_{P}(x,y)$ we always can solve the equation $Pv=u$ modulo a smooth error. In other words, 
$G_{P}(x,y)$ is the kernel function of parametrix for $P$. As such it always exists. 

The Green function $G_{P}(x,y)$ is smooth off the diagonal $y=x$. As this is the kernel function of a 
pseudodifferential operator of order $-m$ (namely, a parametrix for $P$), the theory of pseudodifferential operators enables us 
to describe the form of its singularity near the diagonal. 

\begin{proposition}[See~\cite{Ta:PDE2}]
 Suppose that $P$ has order $m\leq n$. Then, in local coordinates and near the diagonal $y=x$, 
 \begin{equation}
     G_{P}(x,y)= |x-y|^{-n+m}\sum_{0\leq j <n-m}a_{j}(x,\theta)|x-y|^{j}-\gamma_{P}(x)\log |x-y|+\op{O}(1),
     \label{eq:Green.singularity}
 \end{equation}where $\theta=|x-y|^{-1}(x-y)$ and the functions $a_{j}(x,\theta)$ and $\gamma_{P}(x)$ are smooth. 
 Moreover, 
 \begin{equation}
     \gamma_{P}(x)=(2\pi)^{-n}\int_{S^{n-1}}p_{-n}(x,\xi)d^{n-1}\xi,
     \label{eq:Green.log-singularity}
 \end{equation}where $p_{-n}(x,\xi)$ is the symbol of degree $-n$ of any parametrix for $P$. 
\end{proposition}

\begin{remark}
    The various terms in~(\ref{eq:Green.singularity}) depends on the various components of the symbol of a parametrix for $P$, and so, at 
    the exception of the leading term $a_{0}(x,\theta)$, they transform in some cumbersome way under changes of local 
    coordinates. However, it can be shown, that the coefficient $\gamma_{P}(x)$ makes sense intrinsically on $M$.  
\end{remark}

\begin{remark}
  Suppose that $P$ is a differential operator. Then any parametrix for $P$ is odd class in the sense of~\cite{KV:GDEO}. This implies that its symbols are homogeneous with 
   respect to the symmetry $\xi \rightarrow -\xi$, e.g., $p_{-n}(x,-\xi)=(-1)^{n}p_{-n}(x,\xi)$. In particular, when the 
   dimension $n$ is odd, we get 
   \begin{equation*}
       \int_{S^{n-1}}p_{-n}(x,\xi)d^{n-1}\xi=\int_{S^{n-1}}p_{-n}(x,-\xi)d^{n-1}\xi=-\int_{S^{n-1}}p_{-n}(x,\xi)d^{n-1}\xi=0,
   \end{equation*}and hence $\gamma_{P}(x)=0$ for all $x \in M$.
\end{remark}

The formula~(\ref{eq:Green.log-singularity}) provides us with a close relationship with the noncommutative residue trace of 
Guillemin~\cite{Gu:NPWF, Gu:GLD} and Wodzicki~\cite{Wo:LISA, Wo:NCRF}, since the right-hand side of~(\ref{eq:Green.log-singularity}) gives the noncommutative 
residue density of a parametrix for $P$. This noncommutative residue trace  is the residual functional induced on 
integer order \psidos\ by the analytic extension of the ordinary trace to non-integer order \psidos. This enables us to obtain 
the following statement.  


\begin{proposition}\label{eq:Green.zeta-function}
    Assume that $M$ is compact. Then the zeta function $\zeta(P;s)=\Tr P^{-s}$, $\Re s >\frac{n}{m}$, has a meromorphic extension 
    to the whole complex plane $\C$ with at worst simple pole singularities. Moreover, at $s=1$, we have 
    \begin{equation}
     m \Res_{s=1} \Tr P^{-s}=\int_{M}\gamma_{P}(x) v_{g}(x). 
     \label{eq:Green.zeta-log-sing}
    \end{equation}
\end{proposition}

\begin{remark}
    We refer to~\cite{Se:CPEO} for the construction of the complex powers $P^{s}$, $s \in \C$. The construction depends on the choice 
    of a spectral cutting for both $P$ and its principal symbol, but the residues at integer points of the zeta 
    function $\zeta(P;s)=\Tr P^{-s}$ do not depend on this choice. 
\end{remark}

\begin{remark}
    Proposition~\ref{eq:Green.zeta-function} has a local version (see~\cite{Gu:GLD, KV:GDEO, Wo:NCRF}). If we denote by $K_{P^{-s}}(x,y)$ the kernel of $P^{-s}$ for $\Re 
    s>\frac{n}{m}$, then the map $s \rightarrow K_{P^{-s}}(x,x)$ has a meromorphic extension to $\C$ with at worst 
    simple pole singularities in such a way that
    \begin{equation}
      m \Res_{s=1}  K_{P^{-s}}(x,x)=\gamma_{P}(x). 
      \label{eq:Green.residue-local-zeta-gamma}
    \end{equation}We also note that the above equality continues to hold if we replace $P^{-s}$ by any holomorphic 
   \psido\ family $P(s)$ such that $\ord P(s)=-ms$ and $P(1)$ is a parametrix for $P$. This way Eq.~(\ref{eq:Green.residue-local-zeta-gamma}) 
   continues to hold when $M$ is not compact. 
\end{remark}

Suppose now that $M$ is compact and $P$ has a positive leading symbol, so that we can form the heat semigroup $e^{-tP}$ and the heat 
kernel $K_{P}(x,y;t)$ associated to $P$. The heat kernel asymptotics for $P$ then takes the form,
\begin{equation}
    K_{P}(x,x;t) \sim (4\pi t)^{-\frac{n}{m}} \sum_{j\geq 0} t^{\frac{j}{m}}a_{j}(P;x) + \log t \sum_{k\geq 1} t^{k} b_{j}(P;x)\qquad \text{as $t\rightarrow 0^{+}$}. 
    \label{eq:Green.heat-kernel-asymptotics}
\end{equation}
We refer to~\cite{Wi:CSCPDO, GS:WPPDOAPSBP} for a derivation of the above heat kernel asymptotics for \psidos. The residues of the local zeta function $\zeta(P;s;x)(x)=K_{P^{-s}}(x,x)$ are related to the coefficient in the heat kernel asymptotic for $P$ as follows. 

By Mellin's formula, for $\Re s>0$ we have
\begin{equation*}
    \Gamma(s)P^{-s}=\int_{0}^{\infty}t^{s-1}(1-\Pi_{0})e^{-tP}dt, \qquad \Re s>0,
\end{equation*}where $\Pi_{0}$ is the orthogonal projection onto the nullspace of $P$. Together with the heat kernel 
asymptotics~(\ref{eq:Green.heat-kernel-asymptotics}) this implies that, for $\Re s>\frac{n}{m}$, 
\begin{align*}
     \Gamma(s)K_{P^{-s}}(x,x)& =\int_{0}^{1}t^{s-1}K_{P}(x,x;t)dt +h(x;s)   
     \\  & = (4\pi)^{-\frac{n}{m}}\sum_{0\leq j <n} 
     \int_{0}^{1}t^{s+\frac{j-n}{m}}a_{j}(P;x)t^{-1}dt+h(x;s)\\
     & = (4\pi )^{-\frac{n}{m}} \sum_{0\leq j <n} \frac{m}{ms+j-n} a_{j}(P;x)+h(x;s) \op{Hol}(\Re s>0),
\end{align*}where $h(x;s)$ is a general notation for a holomorphic function on the halfspace $\Re s>0$. Thus, for $j=0,1,\ldots, n-1$, we have
\begin{equation*}
    \Gamma\left( \frac{n-j}{m}\right)\Res_{s=\frac{n-j}{m}}K_{P^{-s}}(x,x)=(4\pi )^{-\frac{n}{m}}a_{j}(P;x). 
\end{equation*}
Combining this with~(\ref{eq:Green.residue-local-zeta-gamma}) we arrive at the following result. 

\begin{proposition}[Compare~\cite{PR:ICL}]\label{prop:Green.heat-gamma}
    Let $k\in m^{-1}\N$ be such that $mk-n\in \N_{0}$. Then, under the above assumptions,  we have
    \begin{equation}
        \gamma_{P^{k}}(x)= (4\pi )^{-\frac{n}{m}}\frac{m}{\Gamma(k)}a_{n-mk}(P;x). 
        \label{eq:Green.log-sing-heat}
    \end{equation}
\end{proposition}

The above result provides us with a way to compute the logarithmic singularities of the Green functions of the powers 
of $P$ from the knowledge of coefficients of the heat kernel asymptotics. For instance, in the case of the Laplace 
operator combining Theorem~\ref{thm:Heat.coefficients} and Proposition~\ref{prop:Green.heat-gamma} gives the following result. 

\begin{proposition}For $k=\frac{n}{2},\frac{n}{2}-1,\frac{n}{2}-2$ and provided that $k>0$, we have
   \begin{gather*}
       \gamma_{\Delta_g^{\frac{n}{2}}}(x)= (4\pi )^{-\frac{n}{2}}\frac{2}{\Gamma\left(\frac{n}{2}\right)}, \qquad  
        \gamma_{\Delta_g^{\frac{n}{2}-1}}(x)= (4\pi )^{-\frac{n}{2}}\frac{2}{\Gamma\left(\frac{n}{2}-1\right)}. \frac{-1}{6}\kappa,\\
       \gamma_{\Delta_g^{\frac{n}{2}-2}}(x)=(4\pi )^{-\frac{n}{2}}\frac{2}{\Gamma\left(\frac{n}{2}-2\right)} 
       \left( 
       \frac{1}{180}|R|^{2}-\frac{1}{180}|\op{Ric}|^{2}+\frac{1}{72}\kappa^{2}-\frac{1}{30}\Delta_{g}\kappa\right).
   \end{gather*}
\end{proposition}

\section{Green Functions and Conformal Geometry}\label{sec:Main}
Green functions of the Yamabe operator and other conformal powers of the Laplacian play an important role in conformal 
geometry. This is illustrated by the solution to the Yamabe problem by Schoen~\cite{Sc:CDRMCSC} or by the work of 
Okikoliu~\cite{Ok:CMDLOD} on variation formulas for the zeta-regularized determinant of the Yamabe operator in odd dimension. 
Furthermore, Parker-Rosenberg~\cite{PR:ICL} computed the logarithmic singularity of the Green function of the Yamabe operator in even 
low dimension.

\begin{theorem}[{Parker-Rosenberg~\cite[Proposition~4.2]{PR:ICL}}]\label{thm:Main.PR-Yamabe}
Let $(M^{n},g)$ be a closed Riemannian manifold. Then
 \begin{enumerate}
      
      \item In dimension $n=2$ and $n=4$ we have
     \begin{equation*}
         \gamma_{P_{1,g}}(x)= 2(4\pi)^{-1}\ (n=2)\qquad  \text{and} \qquad \gamma_{P_{1}}(x)=0 \ (n=4).
     \end{equation*}
 
      \item  In  dimension $n=6$, 
     \begin{equation*}
         \gamma_{P_{1,g}}(x)= (4\pi)^{-3} \frac{1}{90}|W|^{2},
     \end{equation*}where $|W|^{2}=W^{ijkl}W_{ijkl}$ is the norm-square of the Weyl tensor.
 
     \item  In dimension $n=8$,
     \begin{equation*}
           \gamma_{P_{1,g}}(x)=(4\pi)^{-4}\frac{2}{9\cdot 7!} \left(
           81\Phi_{g}+ 64 W_{ij}^{\mbox{~~}\mbox{~~}kl}W_{\mbox{~~}\mbox{~~}pq}^{ij}W_{\mbox{~~}\mbox{~~}kl}^{pq}+352
             W_{ijkl} W^{i\mbox{~~}k}_{\;p\mbox{~}\,q} W^{pjql} \right),
      \end{equation*}where $\Phi_{g}$ is given by~(\ref{eq:Ambient.Phi-invariant}). 
 \end{enumerate}
  \end{theorem}  
\begin{remark}
Parker and Rosenberg actually computed the coefficient $a_{\frac{n}{2}-1}(P_{1,g};x)$ of $t^{-1}$ in the heat kernel asymptotics for the Yamabe 
operator when $M$ is closed. We obtain the above formulas for $\gamma_{P_{1,g}}(x)$ by 
using~(\ref{eq:Green.log-sing-heat}). We also note that Parker and Rosenberg used a different sign convention for the 
curvature tensor. 
\end{remark}   
  
The computation of Parker and Rosenberg in~\cite{PR:ICL} has two main steps. The first step uses the formulas 
for Gilkey~\cite{Gi:SGRM, Gi:ITHEASIT} for the heat invariants of Laplace type operator, since the Yamabe operator is such an operator. This 
expresses the coefficient $a_{\frac{n}{2}-1}(P_{1,g};x)$ as a linear combination of Weyl Riemannian invariants. For $n=8$ there are 17 
such invariants. The 2nd step consists in rewriting these linear combinations in terms of Weyl conformal invariant, so as 
to obtain the much simpler formulas above. 

It is not clear how to extend Parker-Rosenberg's approach for computing the logarithmic singularities 
$\gamma_{P_{k,g}}(x)$ of the Green functions of other conformal powers of the Laplacian. This 
would involve computing the coefficients of the heat kernel asymptotics for \emph{all} these operators, including 
conformal fractional powers of the Laplacian in odd dimension.

As the following result shows, somewhat amazingly, in order to compute the $\gamma_{P_{k,g}}(x)$ the \emph{sole} knowledge of 
the coefficients heat kernel asymptotics for the Laplace operator is enough.

\begin{theorem}\label{thm:Main.main}
       Let $(M^{n},g)$ be a Riemannian manifold and let $k\in \frac{1}{2}\N$ be such that $\frac{n}{2}-k\in \N_{0}$. 
       Then the logarithmic singularity of the Green function of $P_{k,g}$ is given by
       \begin{equation*}
           \gamma_{P_{k,g}}(x)= \frac{2}{\Gamma(k)}(4\pi)^{-\frac{n}{2}}
           \tilde{a}_{n-2k}(\Delta_{g};x),
       \end{equation*}where $\tilde{a}_{n-2k}(\Delta_{g};x)$ is the local conformal invariant obtained by applying Fefferman-Graham's rule to the heat invariant 
       ${a}_{n-2k}(\Delta_{g};x)$ in~(\ref{eq:Heat.heat-kernel-asymptotics}).
\end{theorem}
\begin{remark}
  When $n$ is odd the condition $\frac{n}{2}-k\in \N_{0}$ imposes $k$ to be an half-integer, and so in this case $P_{k,g}$ is a 
  conformal fractional powers of the Laplacian.
\end{remark}

Using Theorem~\ref{thm:Main.main} and the knowledge of the heat invariants ${a}_{2j}(\Delta_{g};x)$ it becomes straighforward to compute the 
logarithmic singularities $\gamma_{P_{k,g}}(x)$. In particular, we recover the formulas of 
Parker-Rosenberg~\cite{PR:ICL} stated in Theorem~\ref{thm:Main.PR-Yamabe}.  

\begin{theorem}[\cite{Po:LSSKLICCRS, Po:Clay}]
Let $(M^{n},g)$ be a Riemannian manifold of dimension $n\geq 3$. Then 
\begin{equation*}
    \gamma_{P_{\frac{n}{2},g}}(x)= \frac{2}{\Gamma\left(\frac{n}{2}\right)}(4\pi)^{-\frac{n}{2}}\qquad \text{and} \qquad  
    \gamma_{P_{\frac{n}{2}-1,g}}(x)=0.
\end{equation*}
\end{theorem}
\begin{proof}
 We know from Theorem~\ref{thm:Heat.coefficients} that $a_{0}(\Delta_{g};x)=1$ and $a_{2}(\Delta_{g};x)= -\frac{1}{6}\kappa$. Thus 
 $\tilde{a}_{0}(\Delta_{g};x)=1$. Moreover, as $\kappa$ is the trace of Ricci tensor, from Example~\ref{ex:Ambient.FG-rule}.(a) we see that 
 $\tilde{a}_{2}(\Delta_{g};x)=0$. Combining this with Theorem~\ref{thm:Main.main} gives the result. 
\end{proof}

\begin{remark}
    The equality  $\gamma_{P_{\frac{n}{2}}}(x)= 2\Gamma\left(\frac{n}{2}\right)^{-1}(4\pi)^{-\frac{n}{2}}$ can also be obtained by 
    direct computation using~(\ref{eq:Green.log-singularity}) and the fact that in this case $p_{-n}(x,\xi)$ is the principal symbol of a parametrix for 
    $P_{\frac{n}{2},g}$. As $P_{\frac{n}{2},g}$ has same principal symbol as $\Delta_{g}^{\frac{n}{2}}$, the principal 
    symbol of any  parametrix for  $P_{\frac{n}{2},g}$ is equal to $|\xi|_{g}^{-n}$, where 
    $|\xi|_{g}^{2}=g^{ij}(x)\xi_{i}\xi_{j}$. 
\end{remark}

\begin{remark}
    The equality $\gamma_{P_{\frac{n}{2}-1,g}}(x)=0$ is obtained in~\cite{Po:LSSKLICCRS, Po:Clay} by showing this is a conformal invariant weight $1$ and using the fact there is 
    no nonzero conformal invariant of weight 1. Indeed, by Theorem~\ref{thm:Heat.invariant-theory} any Riemannian invariant of weight 1 is a scalar 
    multiple of the scalar curvature, but the scalar curvature is not a conformal invariant.  Therefore, any local 
    conformal of weight 1 must be zero. 
 \end{remark}

\begin{theorem}\label{thm:Main.weight4}
Let $(M^{n},g)$ be a Riemannian manifold of dimension $n\geq 5$. Then 
\begin{equation}
  \gamma_{P_{\frac{n}{2}-2}}(x)=\frac{(4\pi)^{-\frac{n}{2}}}{\Gamma\left(\frac{n}{2}-2\right)}.\frac{1}{90}|W|^{2}.  
  \label{eq:Main.weight2}
\end{equation}
\end{theorem}
\begin{proof}
    By Theorem~\ref{thm:Heat.coefficients} we have 
    \[a_{4}(\Delta_{g};x)=\frac{1}{180}|R|^{2}-\frac{1}{180}|\op{Ric}|^{2}+\frac{1}{72}\kappa^{2}-\frac{1}{30}\Delta_{g}\kappa.\]
    As Example~\ref{ex:Ambient.FG-rule}.(b) shows, the Weyl conformal invariant 
    associated to $|R|^{2}$ by Fefferman-Graham's rule is $|W|^{2}$. Moreover, as $|\op{Ric}|^{2}$, $\kappa^{2}$, and 
    $\Delta_{g}\kappa$ involve the Ricci tensor, the corresponding Weyl conformal invariants are equal to 0. Thus,
    \[\tilde{a}_{4}(\Delta_{g};x)= \frac{1}{180}|W|^{2}.\] Applying Theorem~\ref{thm:Main.main} then completes the proof. 
\end{proof}

\begin{theorem}\label{thm:Main.weight6}
Let $(M^{n},g)$ be a Riemannian manifold of dimension $n\geq 7$. Then 
\begin{equation*}
     \gamma_{P_{\frac{n}{2}-3}}(x)=\frac{(4\pi)^{-\frac{n}{2}}}{\Gamma\left(\frac{n}{2}-3\right)}\cdot \frac{2}{9\cdot 7!} \left(
           81\Phi_{g}+ 64 W_{ij}^{\mbox{~~}\mbox{~~}kl}W_{\mbox{~~}\mbox{~~}pq}^{ij}W_{\mbox{~~}\mbox{~~}kl}^{pq}+352
             W_{ijkl} W^{i\mbox{~~}k}_{\;p\mbox{~}\,q} W^{pjql}
            \right),
            \label{eq:Main.weight3}
\end{equation*}where $\Phi_{g}(x)$ is given by~(\ref{eq:Ambient.Phi-invariant}). 
\end{theorem}
\begin{proof}
   Thanks to Theorem~\ref{thm:Heat.coefficients} we know that  
   \begin{equation*}
    a_{6}(\Delta_{g};x)=     \frac{1}{9\cdot 7!} \left( 81|\nabla R|^{2} + 
    64 R_{ij}^{\mbox{~~}\mbox{~~}kl}R_{\mbox{~~}\mbox{~~}pq}^{ij}R_{\mbox{~~}\mbox{~~}kl}^{pq} + 
    352 R_{ijkl} R^{i\mbox{~~}k}_{\;p\mbox{~}q} R^{pjql} \right) +I_{g,0}^{(3)}(x),
   \end{equation*}where $I_{g,0}^{(3)}(x)$ is a linear combination of Weyl Riemannian invariants involving the Ricci 
   tensor. Thus $\tilde{a}_{6}(\Delta_{g};x)=0$ by Example~\ref{ex:Ambient.FG-rule}.(a). By Example~\ref{ex:Ambient.FG-rule}.(b), the Weyl conformal 
   invariants corresponding to $R_{ij}^{\mbox{~~}\mbox{~~}kl}R_{kl}^{\mbox{~~}\mbox{~~}pq}R_{pq}^{\mbox{~~}\mbox{~~}ij}$ and $  R^{i\mbox{~~}\;q}_{\,jk} R_{i\mbox{~~}\mbox{~~}\;l}^{\mbox{~~}\,pk} 
            R_{j\mbox{~~}\mbox{~}q}^{\mbox{~~}\,pl}$ are $W_{ij}^{\mbox{~~}\mbox{~~}kl}W_{kl}^{\mbox{~~}\mbox{~~}pq}W_{pq}^{\mbox{~~}\mbox{~~}ij}$ and $W^{i\mbox{~~}\;q}_{\,jk} W_{i\mbox{~~}\mbox{~~}\;l}^{\mbox{~~}pk} 
            W_{j\mbox{~~}\mbox{~}q}^{\mbox{~~}pl}$ respectively. Moreover, by Example~\ref{ex:Ambient.FG-rule}.(c) applying Fefferman-Graham's rule to $|\nabla R|^{2}$ 
            yields the conformal invariant $\Phi_{g}$ given by~(\ref{eq:Ambient.Phi-invariant}). Thus, 
            \begin{equation*}
                \tilde{a}_{6}(\Delta_{g};x)=\frac{1}{9\cdot 7!} \left(
           81\Phi_{g} + 64 W_{ij}^{\mbox{~~}\mbox{~~}kl}W_{\mbox{~~}\mbox{~~}pq}^{ij}W_{\mbox{~~}\mbox{~~}kl}^{pq}+352
             W_{ijkl} W^{i\mbox{~~}k}_{\;p\mbox{~}\,q} W^{pjql}
            \right).
            \end{equation*}Combining this with Theorem~\ref{thm:Main.main} proves the result. 
\end{proof}

\begin{remark}
    In the case of Yamabe operator (i.e., $k=1$) the above results give back the results of 
    Parker-Rosenberg~\cite{PR:ICL} as stated 
    in Theorem~\ref{thm:Main.PR-Yamabe}. 
\end{remark}

Let us present some applications of the above results. Recall that a Riemannian manifold $(M^{n},g)$ is said to be locally conformally flat when it is locally 
conformally equivalent to the flat Euclidean space $\R^{n}$. As is well known, in dimension $n\geq 4$, local conformal flatness is equivalent to 
the vanishing of the Weyl tensor. Therefore, as an immediate consequence of Theorem~\ref{thm:Main.weight4} we obtain the following 
result. 

\begin{theorem}\label{thm:App.local-comformal-flat}
 Let $(M^{n},g)$ be a Riemannian manifold of dimension $n\geq 5$. Then the following are equivalent:
 \begin{enumerate}
     \item $(M^{n},g)$ is locally conformally flat.
 
     \item  The logarithmic singularity $\gamma_{P_{\frac{n}{2}-2},g}(x)$ vanishes identically on $M$.
 \end{enumerate}
\end{theorem}
\begin{remark}
    A well known conjecture of Radamanov~\cite{Ra:CBCnBK} asserts that, for a strictly pseudoconvex domain $\Omega\subset 
    \C^{n}$, the vanishing of the logarithmic singularity of the Bergman kernel is equivalent to $\Omega$ being 
    biholomorphically equivalent to the unit ball $\mathbb{B}^{2n}\subset \C^{n}$. Therefore, we may see Theorem~\ref{thm:App.local-comformal-flat} 
    as an analogue of Radamanov conjecture in conformal geometry in terms of Green functions of conformal powers of the 
    Laplacian. 
\end{remark}

In the compact case, we actually obtain a spectral theoretic characterization of the conformal class of the round sphere as follows.

\begin{theorem}\label{thm:App.spectral-caract-sphere}
 Let $(M^{n},g)$ be a compact simply connected Riemannian manifold of dimension $n\geq 5$. Then the following are equivalent:
 \begin{enumerate}
            \item  $(M^{n},g)$ is conformally equivalent to the round sphere $\mathbb{S}^{n}$.
        
            \item  $\int_{M}\gamma_{P_{\frac{n}{2}-2},g}(x)v_{g}(x)=0$.
        
            \item  $\op{Res}_{s=1}\Tr \left[ (P_{\frac{n}{2}-2})^{-s}\right]=0$.
 \end{enumerate}
\end{theorem}
\begin{proof}
    The equivalence between (2) and (3) is a consequence of~(\ref{eq:Green.zeta-log-sing}). By Theorem~\ref{thm:Main.weight4}, we have
    \begin{equation*}
       \int_{M}\gamma_{P_{\frac{n}{2}-2},g}(x)v_{g}(x)=\frac{(4\pi)^{-\frac{n}{2}}}{\Gamma\left(\frac{n}{2}-2\right)}.\frac{1}{90}\int_{M}|W(x)|^{2}v_{g}(x).
    \end{equation*}
    As $|W(x)|^{2}\geq 0$ we see that $\int_{M}\gamma_{P_{\frac{n}{2}-2},g}(x)v_{g}(x)=0$ if and only if $W(x)=0$ 
    identically on $M$. As $M$ is compact and simply connected, a well known result of Kuiper~\cite{Ku:CFSL} asserts that 
    the vanishing of Weyl tensor is equivalent to the existence of a conformal diffeomorphism from $(M,g)$ onto the 
    round sphere $\mathbb{S}^{n}$. Thus (1) and (2) are equivalent. The proof is complete.  
\end{proof}

\section{Proof of Theorem~\ref{thm:Main.main}}\label{sec:Outline}

In this section, we outline the proof of Theorem~\ref{thm:Main.main}. The strategy of the proof is divided into 9 main steps. 

\begin{step-main}\label{step:Riemannian}
 Let $k\in \frac{1}{2}\N$ be such that $n-2k\in \N_{0}$. Then $\gamma_{\Delta_{g}^{k}}(x)= 2\Gamma(k)^{-1}(4\pi)^{-\frac{n}{2}}{a}_{n-2k}(\Delta_{g};x)$.
\end{step-main}

This follows from Proposition~\ref{prop:Green.heat-gamma}. This is the Riemannian version of Theorem~\ref{thm:Main.main} and the main impetus for that result. 

\begin{step-main}\label{step:conformal-invariance}
 Let $k\in \frac{1}{2}\N$ be such that $n-2k\in \N_{0}$. Then  $\gamma_{P_{k,g}}(x)$ is a  linear combination of Weyl conformal invariants of weight $n-2k$.
\end{step-main}

This step is carried out in~\cite{Po:LSSKLICCRS, Po:Clay}. This is a general result for the logarithmic singularities 
of the Green functions of conformally invariant \psidos. For the Yamabe operator, the transformation of  
$\gamma_{P_{k,g}}(x)$ under conformal change of metrics was observed by 
Parker-Rosenberg~\cite{PR:ICL}. Their argument extends to general conformally invariant operators. 

\begin{step-main}\label{step:invariants-Ricci-flat}
    Let $I_{g}(x)$ be a local Riemannian invariant of any weight $w\in 2\N_{0}$ if $n$ is odd or of weight $w\leq n$ if 
    $n$ is even. If $(M,g)$ is Ricci-flat, then 
    \begin{equation*}
        I_{g}(x)=\tilde{I}_{g}(x) \qquad \text{on $M$},
    \end{equation*}where $\tilde{I}_{g}(x)$ is the local conformal invariant associted to $I_{g}(x)$ by Fefferman-Graham's rule.
\end{step-main}
 This results seems to be new. It uses the fact that if $(M,g)$ is Ricci flat, then by Example~\ref{ex:Ambient.Einstein} the ambient metric is given by
\begin{equation*}
 \tilde{g}=2\rho (dt)^{2}+t^{2}g_{ij}(x)+2tdtd\rho. 
\end{equation*}

\begin{step-main}\label{step:Ricci-flat}
 Theorem~\ref{thm:Main.main} holds when $(M^{n},g)$ is Ricci-flat. 
\end{step-main}
If $(M,g)$ is Ricci-flat, then it follows from Example~\ref{lem:GJMS.GJMS-Einstein} and Example~\ref{ex:scattering.ricci-flat-Pzg} that 
$P_{k,g}=\Delta_{g}^{k}$, and hence $\gamma_{P_{k,g}}(x)=\gamma_{\Delta_{g}^{k}}(x)$. Combining this with Step~\ref{step:Riemannian}  and 
Step~\ref{step:invariants-Ricci-flat} then gives
\begin{equation*}
   \gamma_{P_{k,g}}(x)=\gamma_{\Delta_{g}^{k}}(x)= 2\Gamma(k)^{-1}(4\pi)^{-\frac{n}{2}}{a}_{n-2k}(\Delta_{g};x)= 
   2\Gamma(k)^{-1}(4\pi)^{-\frac{n}{2}}\tilde{a}_{n-2k}(\Delta_{g};x).
\end{equation*}This proves Step~\ref{step:Ricci-flat}. 

In order to prove Theorem~\ref{thm:Main.main} for general Riemannian 
metrics we actually need to establish pointwise versions of Step~\ref{step:Ricci-flat} and Step~\ref{step:invariants-Ricci-flat}. This is the purpose of the next three steps. 

\begin{step-main}\label{step:asymptotically-Ricci-flat-invariant}
    Let $I_{g}(x)=\op{Contr}_{g}\left( \nabla^{k_{1}}R\otimes \cdots \otimes \nabla^{k_{l}}R\right)$ be a Weyl Riemannian invariant of weight $<n$ 
    and assume that $\Ric(g)=\op{O}(|x-x_{0}|^{n-2})$ near a 
    point $x_{0}\in M$. Then 
    $\tilde{I}_{g}(x)={I}_{g}(x)$ at $x=x_{0}$. 
\end{step-main}
 
This step follows the properties of the ambient metric and complete metric contractions of  ambient 
curvature tensors described in~\cite{FG:AM}. 

\begin{step-main}\label{step:asymptotically-Ricci-flat-gamma}
Let $k\in \frac{1}{2}\N$ be such that $n-2k\in 2\N_{0}$ and assume that $\Ric(g)=\op{O}(|x-x_{0}|^{n-2})$ near a 
    point $x_{0}\in M$. Then $\gamma_{P_{k,g}}(x)=\gamma_{\Delta_{g}^{k}}(x)$ at $x=x_{0}$. 
\end{step-main}

If $P$ is an elliptic \psido,  then~(\ref{eq:Green.log-singularity}) gives an expression in local coordinates for $\gamma_{P}(x)$ 
in terms of the symbol of degree $-n$ of a parametrix for $P$. The construction of the symbol of a parametrix shows 
that this symbol is a polynomial in terms of inverse of the principal symbol of $P$ and partial derivatives of its other
homogeneous symbols. It then follows that if $P$ is Riemannian invariant, then $\gamma_{P}(x)$ is of the form~(\ref{eq:Heat.Riemann-invariant}), 
that is, this is a Riemannian invariant (see~\cite{Po:LSSKLICCRS, Po:Clay}). 

Moreover, it  follows from the ambient metric construction of the GJMS operator~\cite{GJMS:CIPLIE, FG:AM} that, when $k$ is an integer, $P_{k,g}$ and $\Delta_{g}^{k}$ 
differs by a differential operator whose coefficients are polynomials in the covariant derivatives of 
order~$<n-2$ of the Ricci tensor. As a result $\gamma_{P_{k,g}}(x)$ and $\gamma_{\Delta_{g}^{k}}(x)$ differ by a linear combination of 
complete metric contraction of tensor powers of covariant derivatives of the Ricci tensor. Thus $\gamma_{P_{k,g}}(x)$ and $\gamma_{\Delta_{g}^{k}}(x)$ agree at 
$x=x_{0}$ if $\Ric(g)$ vanishes at order~$<n-2$ near $x_{0}$. This is also true for the conformal fractional powers of 
the Laplacian thanks to the properties of the scattering matrix for Poincar\'e-Einstein metric in~\cite{JS:ISAHM} and~\cite{GZ:SMCG}. 

%

\begin{step-main}\label{step:Ricci-flat-invariants}
    Let $I_{g}(x)=\op{Contr}_{g}\left( \nabla^{k_{1}}R\otimes \cdots \otimes \nabla^{k_{l}}R\right)$ be a Weyl 
    Riemannian invariant of weight $w< n$. Then the following are equivalent:
    \begin{enumerate}
        \item[(i)] $\tilde{I}_{g}(x)$ at $x=0$ for all metrics of signature $(n,0)$ on $\R^{n}$.  
    
        \item[(ii)] $I_{g}(x)=0$ at $x=0$ for all metrics of signature $(n,0)$ on 
        $\R^{n}$ such that $\Ric(g)=\op{O}\left(|x|^{n-2}\right)$ near $x=0$.  
    
        \item[(iii)] $I_{g}(x)=0$ at $x=0$ for all metrics of signature $(n+1,1)$ on 
        $\R^{n+2}$ such that $\Ric(g)=\op{O}\left(|x|^{n-2}\right)$ near $x=0$.  
    \end{enumerate}
\end{step-main}

The implication $\text{(i)}  \Rightarrow \text{(ii)}$ is an immediate consequence of Step~\ref{step:asymptotically-Ricci-flat-invariant}. 
The implication $\text{(iii)} \Rightarrow  \text{(i)}$ follows from the construction of the conformal 
invariant $\tilde{I}_{g}(x)$, since the ambient metric is of the form given in (iii) if we use the variable 
$r=\sqrt{2\rho}$ instead of $\rho$ (cf.~\cite[Chapter 4]{FG:AM}; see in particular Theorem~4.5 and Proposition~4.7). 
Therefore, the bulk of Step~\ref{step:Ricci-flat-invariants} is proving that (ii) implies (iii).

It can be given sense to a vector space structure on formal linear combinations of complete metric contractions of tensor without 
any reference to dimension or the  signature of the metric. If the weight is less than $n$, then the formal vanishing is equivalent 
to the algebraic vanishing by inputing tensors in dimension $n$. The idea in the proof of the proof of the implication 
$\text{(ii)}  \Rightarrow  \text{(iii)}$ is showing that (ii) and (iii) are equivalent to the same system of linear equations on the space of formal linear combinations of 
complete metric contractions of Ricci-flat curvature tensors. This involves proving a version of the ``2nd main theorem 
of invariant theory'' for Ricci-flat curvature tensors, rather than for collections of trace-free tensors satisfying 
the Young symmetries of Riemannian curvature tensors. In other words, we 
need to establish a \emph{nonlinear} version of~\cite[Theorem~B.3]{BEG:ITCCRG}. 

\begin{step-main}
    Proof of Theorem~\ref{thm:Main.main}.
\end{step-main}

By Step~\ref{step:conformal-invariance} we know that $\gamma_{P_{k,g}}(x)$ is a linear combination of Weyl conformal invariants of weight~$2n-k$, 
so by Theorem~\ref{prop:Ambient.invariant-theory} there is a Riemannian invariant $I_{g}(x)$ of weight $2n-k$ such that
\begin{equation*}
    \gamma_{P_{k,g}}(x)=2\Gamma(k)^{-1}(4\pi)^{-\frac{n}{2}}\tilde{I}_{g}(x).
\end{equation*}
We need to show that $I_{g}(x)={a}_{n-2k}(\Delta_{g};x)$. If $(M,g)$ is Ricci flat, then by 
Step~\ref{step:invariants-Ricci-flat} and Step~\ref{step:Ricci-flat},
\begin{equation*}
    I_{g}(x)=\tilde{I}_{g}(x)= 
   \frac12\Gamma(k)(4\pi)^{\frac{n}{2}} \gamma_{P_{k,g}}(x)= \frac12\Gamma(k)(4\pi)^{\frac{n}{2}}\gamma_{\Delta_{g}^{k}}(x)={a}_{n-2k}(\Delta_{g};x).
\end{equation*}By Step~\ref{step:asymptotically-Ricci-flat-gamma} this result continues to hold at a point $x_{0}\in M$ if 
$\Ric(g)=\op{O}(|x-x_{0}|^{n-2})$ near $x_{0}$. Therefore, in this case the Riemannian invariant  $ 
I_{g}(x)-{a}_{n-2k}(\Delta_{g};x)$ vanishes at $x_{0}$. Using Step~\ref{step:Ricci-flat-invariants} we then 
deduce that the associated conformal invariant vanishes. That is, 
\begin{equation*}
    \gamma_{P_{k,g}}(x)=2\Gamma(k)^{-1}(4\pi)^{-\frac{n}{2}}\tilde{a}_{n-2k}(\Delta_{g};x).
\end{equation*}This proves Theorem~\ref{thm:Main.main}.

\end{document}